\documentclass{amsart}
\usepackage[english]{babel}
\usepackage[latin1]{inputenc}
\usepackage[dvips,final]{graphics}
\usepackage{amsmath,amsfonts,amssymb,amsthm,amscd,array,
stmaryrd,mathrsfs,mathdots,epigraph}
\usepackage[makeroom]{cancel}
\usepackage{pstricks}
\usepackage[all]{xy}
\usepackage{url}
\usepackage{multirow, blkarray}
\usepackage{booktabs}
\usepackage{textcomp}
\usepackage[final]{epsfig}
\usepackage{color}
\usepackage{mathabx}
\theoremstyle{plain}
\newtheorem{thm}{Theorem}
\newtheorem{lem}{Lemma}[section]
\newtheorem{prop}[lem]{Proposition}
\newtheorem{cor}[lem]{Corollary}

\theoremstyle{definition}
\newtheorem*{defn}{Definition}
\newtheorem*{rem}{Remark}
\newtheorem*{rems}{Remarks}
\newtheorem*{ex}{Example}
\newtheorem*{exs}{Examples}

\let\ssection=\section
\renewcommand{\section}{\setcounter{equation}{0}\ssection}

\newcommand{\R}{\mathbb{R}}
\newcommand{\Z}{\mathbb{Z}}
\newcommand{\C}{\mathbb{C}}

\newcommand{\bP}{\mathbb{P}}

\newcommand{\K}{\mathbb{K}}
\newcommand{\RP}{{\mathbb{RP}}}

\newcommand{\CP}{{\mathbb{CP}}}

\newcommand{\DE}{\mathcal{E}}

\newcommand{\cI}{\mathcal{I}}
\newcommand{\cK}{\mathcal{K}}
\newcommand{\cL}{\mathcal{L}}
\newcommand{\cM}{\mathcal{M}}

\newcommand{\W}{\mathcal{W}}

\newcommand{\pf}{\mathrm{pf}}
\newcommand{\Id}{\mathrm{Id}}
\newcommand{\SL}{\mathrm{SL}}
\newcommand{\PSL}{\mathrm{PSL}}
\newcommand{\Sp}{\mathrm{Sp}}

\newcommand{\half}{\frac{1}{2}}

\def\e{\varepsilon}

\def\om{\omega}

\def\l{\lambda}

\def\GCD{\mathop{\rm GCD}\nolimits}
\def\Lc{Lagrangian configuration}
\def\mid{\mathop{\rm mid}\nolimits}
\def\mod{\mathop{\rm mod}\nolimits}
\def\ndup{\mathop{\rm nd}\nolimits}
\def\ocL{\overline \cL}
\def\scr{symplectic cross-ratio}
\def\sgn{\mathop{\rm sgn}\nolimits}
\def\SLHS{L}
\def\stup{\mathop{\rm st}\nolimits}
\def\thup{\mathop{\rm th}\nolimits}

\begin{document}

\title[Lagrangian configurations]{Lagrangian configurations and symplectic cross-ratios}

\author{Charles H.\ Conley}
\address{
Charles H.\ Conley,
Department of Mathematics 
\\University of North Texas 
\\Denton TX 76203, USA} 
\email{conley@unt.edu}

\author{Valentin Ovsienko}
\address{
Valentin Ovsienko,
CNRS,
Laboratoire de Math\'ematiques 
U.F.R. Sciences Exactes et Naturelles 
Moulin de la Housse - BP 1039 
51687 REIMS cedex 2,
France}
\email{valentin.ovsienko@univ-reims.fr}



\begin{abstract}
We consider moduli spaces of cyclic configurations of $N$ lines
in a $2n$-dimensional symplectic vector space, such that
every set of $n$ consecutive lines generates a Lagrangian subspace.
We study geometric and combinatorial problems
related to these moduli spaces,
and prove that they are isomorphic to quotients of spaces
of symmetric linear difference operators with monodromy~$-1$.

The symplectic cross-ratio is an invariant of two pairs
of $1$-dimensional subspaces of a symplectic vector space.
For $N = 2n+2$, the moduli space of \Lc s
is parametrized by $n+1$ symplectic cross-ratios.
These cross-ratios satisfy a single remarkable relation,
related to tridiagonal determinants and continuants,
given by the Pfaffian of a Gram matrix.
\end{abstract}

\maketitle

\thispagestyle{empty}

\tableofcontents

\section{Introduction}

Throughout this article, $\K$ will denote either $\R$ or $\C$,
$\{e_1, \ldots, e_n,\ f_1, \ldots, f_n\}$ will be
the standard basis of $\K^{2n}$, and $\omega$ will be
the standard symplectic form on $\K^{2n}$: for $1 \le i,j \le n$,
\begin{equation*}
   \omega(e_i, e_j) = 0, \qquad
   \omega(e_i, f_j) = \delta_{ij}, \qquad
   \omega(f_i, f_j) = 0.
\end{equation*}
We define the symplectic group $\Sp(2n, \K)$ with respect to $\omega$.
Our topic of study is configurations of 1-dimensional subspaces
of $\K^{2n}$ modulo the action of $\Sp(2n, \K)$.

\begin{defn}
For $N \ge 2n$, define an {\it $(n, N)$-Lagrangian configuration\/} over $\K$
to be a cyclically ordered $N$-tuple $(X_1, \ldots, X_N)$ of 
lines through the origin in the symplectic space $\K^{2n}$
with the following properties
(by cyclically ordered, we mean that the indices are read modulo~$N$):

\begin{enumerate}
\item[(i)]
Every $n$ consecutive lines span a Lagrangian subspace:
$\langle X_{i+1}, \ldots, X_{i+n} \rangle$ is Lagrangian for all~$i$.

\smallbreak \item[(ii)]
Every $2n$ consecutive lines span the entire symplectic space:
$\langle X_{i+1}, \ldots, X_{i+2n} \rangle = \K^{2n}$ for all~$i$.

\end{enumerate}
\Lc s in the same orbit of $\Sp(2n, \K)$ are said to be {\em equivalent.}
\end{defn}

We have formulated our results in the symplectic setting,
but we may just as well speak of {\em Legendrian configurations\/}
of points in the contact projective space $\K\bP^{2n-1}$.
Let us mention that Legendrian configurations in $\R\bP^3$
may be viewed as discrete analogs of {\it Legendrian knots.\/}
In another direction, one may consider
cyclically ordered $N$-tuples of Lagrangian subspaces
$(L_1, L_2, \ldots, L_N)$ in $\R^{2n}$
such that every two consecutive subspaces
$L_i$ and $L_{i+1}$ are ``maximally non-transversal''.
These configurations are in some sense dual
to Lagrangian configurations, and the {\it Maslov index\/}
may be applied to study them; see~\cite{Arn}.
Continuous versions were treated in~\cite{Ovs}.

Suppose that an arbitrary cyclically ordered $N$-tuple
$(X_1, \ldots, X_N)$ of lines through the origin in $\K^{2n}$
has Property~(i) above.
We will see in Lemma~\ref{ii from i} that then
it has Property~(ii) if and only if the subspace
$\langle X_i, X_{i+n} \rangle$ is not isotropic for any~$i$.
It turns out that $\Sp(2n, \K)$ does not act freely
on all Lagrangian configurations, but for $N > 2n$,
it does act freely on configurations in which $\langle X_i, X_j \rangle$
is not isotropic except when forced to be so by Property~(i);
see Proposition~\ref{freely}.  We refer to such configurations as generic:

\begin{defn}
An $(n,N)$-Lagrangian configuration $\left( X_1,\ldots,X_N\right)$
is {\em generic\/} if $\langle X_i, X_j \rangle$ is non-isotropic
whenever the $N$-cyclic distance between~$i$ and~$j$
is at least~$n$, that is, $i-j$ is not congruent to any of
$0, \pm 1, \pm 2, \ldots, \pm (n-1)$ modulo~$N$.
We denote the $\Sp(2n, \K)$-moduli space of generic
$(n, N)$-configurations over $\K$ by $\cL_{n, N}(\K)$.
\end{defn}

The space $\cL_{n, N}(\K)$ is the main object of our study.
We will see that $\Sp(2n, \K)$ acts freely on generic configurations,
implying that $\cL_{n, N}(\K)$ is a variety of dimension $n(N - 2n - 1)$,
and moreover, $\cL_{n, N}(\R)$ and $\cL_{n, N}(\C)$ are smooth
real and complex manifolds of this dimension, respectively.
We construct a collection of {\it symplectic cross-ratios\/} which are
$\Sp(2n, \K)$-invariants of Lagrangian configurations over~$\K$,
and in some cases show that these cross-ratios 
form a coordinate ring on $\cL_{n, N}(\K)$.
They satisfy certain relations, which we calculate explicitly
as Pfaffians for $N = 2n + 2$, the simplest non-trivial case.
These Pfaffians are closely related to the classical determinants
of continued fractions known as {\it continuants;\/} see \cite{CoOv}.

Observe that for $n = 1$ the Lagrangian condition
is vacuous, so $\cL_{1, N}(\K)$ is essentially the
classical moduli space $\cM_{0,N}$
of configurations of $N$ points on the projective line.
We regard $\cL_{n,N}(\K)$
as a multi-dimensional symplectic variant of $\cM_{0,N}$.
The only previously studied configurations in symplectic space we know of
are triangles and skew lines; see \cite{Yag} and Section~2.8 of \cite{OS}.
To the best of our knowledge, Lagrangian configurations have not been
considered before.  We believe that they deserve further study;
in particular, it would be interesting to investigate the topology of $\cL_{n,N}(\K)$.
It seems plausible that the topological invariants of Legendrian knots,
for instance, the Maslov class and the Bennequin invariant,
as well as more general invariants, can be expressed
in terms of cross-ratios of a Lagrangian configuration.

Relations to dynamical systems also seem promising.
Moduli spaces of cyclic configurations of points in $\RP^n$
(without any Legendrian condition)
carry a family of discrete integrable systems,
including the pentagram map and its generalizations;
see \cite{Glo, GP, KS, OST, OST1, Sol}.
We believe that $\cL_{n,N}(\K)$ also supports
interesting discrete dynamical systems.

\subsection{Example: hexagons in $\K\bP^3$} \label{Hexagons}

Our main geometric result is a description of
$\cL_{n, 2n+2}$.  The simplest case, $\cL_{1, 4}(\K)$,
is the moduli space of quadrilaterals in $\K\bP^1$.
It has been known since ancient times to be
1-dimensional and parametrized by the classical cross-ratio.
Therefore the first new case is $\cL_{2, 6}(\K)$,
the $\Sp(4, \K)$-moduli space of generic
$(2, 6)$-\Lc s, i.e., Legendrian hexagons in $\K\bP^3$.

Given any $(2, 6)$-configuration, choosing a non-zero point
on each of the six lines gives a hexagon $(x_0,\ldots,x_5)$ in $\K^{4}$.
It turns out to be natural to regard this hexagon
as the $6$-antiperiodic sequence $(x_i)_{i \in \Z}$
defined by $x_{i \pm 6} := -x_i$.  Then
$$
\omega(x_i, x_{i+3}) = \omega(x_{i+3}, x_{i+6}),
\qquad
\omega(x_i, x_{i+2}) = \omega(x_{i+2}, x_{i+6}).
$$
The sequences $\bigl(\omega(x_i, x_{i+2})\bigr)_{i \in \Z}$
and $\bigl(\omega(x_i, x_{i+3})\bigr)_{i \in \Z}$
are $6$-periodic and $3$-periodic, respectively.

The Lagrangian conditions are $\omega(x_i, x_{i+1}) = 0$
and $\omega(x_i, x_{i+2}) \not= 0$.
Thus we may say that we are considering hexagons whose
sides are of ``symplectic length zero'', but whose
``symplectic subdiameters'' are non-zero.
The generic configurations are those with non-zero
``symplectic diameters'': $\omega(x_i, x_{i+3}) \not= 0$.

The three ``diametric symplectic cross-ratios''
$$
c_i :=
\frac
{\om(x_i, x_{i+3})\,\om(x_{i+1}, x_{i+4})}
{\om(x_i, x_{i+4})\,\om(x_{i+1}, x_{i+3})},
\quad i = 0,\, 1,\, 2,
$$
depend only on the original configuration of lines,
not the choice of points $x_i$, and are symplectic invariants.
We will see that they form an essentially complete set of invariants
parametrizing $\cL_{2,6}(\K)$.
As an example, we illustrate $c_0$ by the following diagram.
\begin{equation*}
 \xymatrix @!0 @R=0.55cm @C=1.1cm
 {
&x_0\ar@{-}[rd]\ar@{-}[ld]\ar@{-}[lddd]\ar@2{-}[dddd]&
\\
x_5\ar@{-}[dd]&& x_1\ar@{-}[dd]\ar@2{-}[lldd]\ar@{-}[lddd]\\
\\
x_4&& x_2\\
&x_3\ar@{-}[lu]\ar@{-}[ru]&
}
\end{equation*}
\begin{center}
Figure 1.
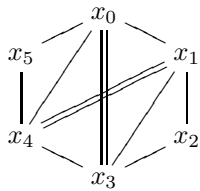
The cross-ratio $c_0=
\frac
{\om(x_0, x_3)\,\om(x_1, x_4)}
{\om(x_0, x_4)\,\om(x_1, x_3)}$ on $\cL_{2,6}(\K)$.
\end{center}

\medbreak
Observe that the space of all $(2, 6)$-\Lc s is
12-dimensional: there are three degrees of freedom
for each of the six points, and six Lagrangian conditions.
As mentioned earlier, the 10-dimensional group $\Sp(4, \K)$
acts freely on the generic configurations,
so $\cL_{2,6}(\K)$ is 2-dimensional.
Therefore the three cross-ratios cannot be independent.
In fact, they satisfy the relation
\begin{equation} \label{FirstEq}
   \frac{1}{c_0} + \frac{1}{c_1} + \frac{1}{c_2} = 1.
\end{equation}
Theorem~\ref{Main Thm} resolves the situation,
describing $\cL_{2,6}(\K)$ completely:

\begin{itemize}
\item
(\ref{FirstEq}) is the only relation on the cross-ratios:
any three non-zero $\K$-scalars $c_0, c_1, c_2$
satisfying it are the cross-ratios of a generic \Lc.

\smallbreak \item
The cross-ratios are complete continuous invariants for \Lc s:
equivalent configurations have the same cross-ratios,
and any two configurations with the same cross-ratios
are equivalent if $\K = \C$, and either equivalent or opposite
(see Section~\ref{OCSI}) if $\K = \R$.

\end{itemize}

These results may be reformulated in terms of
normalized configurations as follows.
For $\K = \C$, the $x_i$ can be rescaled so that
the symplectic subdiameters $\om(x_i,x_{i+2})$ are all~$1$.
Then the symplectic diameters $a_i:=\om(x_i,x_{i+3})$
become, up to an overall choice of sign, symplectic invariants.
Indeed, here $c_i = a_i a_{i+1}$, so~(\ref{FirstEq}) becomes
\begin{equation} \label{FirstEqBis}
   a_0 a_1 a_2 = a_0 + a_1 + a_2.
\end{equation}

For $\K = \R$, it may happen that only complex rescalings
can bring all subdiameters to~$1$.
However, the required scale factors are always
either real or pure imaginary.  There are four possibilities:
the normalized $x_i$ with $i$ even are either all real or all
pure imaginary, and similarly for $i$ odd.
If the normalized $x_i$ are all real or all imaginary,
then the $a_i$ are all real, while if the normalized $x_i$
are half real and half imaginary, then the $a_i$ are all imaginary.

Let us remark that up to permutation, $(c_0, c_1, c_2) = (2, 3, 6)$
is the only {\it Egyptian fraction\/} solution of~(\ref{FirstEq}).
It arises from $(a_0, a_1, a_2) = (1, 2, 3)$, the only
positive integer solution of~(\ref{FirstEqBis}).
Integer solutions of the multi-dimensional analogs of these relations
are discussed in \cite{CoOv, Ovs1}.

\subsection{Example: the Gauss relations} \label{Gauss relations}

In the final section of this article we make some
initial remarks on the relations between
the symplectic cross-ratios of $\cL_{n, 2n + 3}$.
Historically, the earliest examples of relations between
cross-ratios arose in Gauss' {\it pentagramma mirificum} \cite{Gau},
which is $\cM_{0, 5}(\K)$, or equivalently, $\cL_{1, 5}(\K)$,
the moduli space of pentagons in $\K\bP^1$.

As we did for Legendrian hexagons, given five points in $\K\bP^1$,
lift them to non-zero points $x_0, \cdots, x_4$ in $\K^2$
and extend to a 5-antiperiodic sequence
$(x_i)_{i \in \Z}$ via $x_{i \pm 5} := -x_i$.
Gauss discovered that the 5-periodic sequence of cross-ratios
$d_i := \omega(x_{i-3}, x_i) \omega(x_{i-2}, x_{i-1}) /
\omega(x_{i-3}, x_{i-2}) \omega(x_{i-1}, x_i)$ satisfy the relations
\begin{equation} \label{Penta}
   d_i d_{i+1} = d_{i+3} + 1.
\end{equation}

These five {\em Gauss relations\/} completely determine
the varietal structure of $\cM_{0,5}$.
They can be rewritten in the remarkable form
$$
\biggl(
\begin{array}{cc}
d_0 & 1 \\[4pt]
-1 & 0
\end{array}
\biggr)
\biggl(
\begin{array}{cc}
d_1 & 1 \\[4pt]
-1 & 0
\end{array}
\biggr)
\biggl(
\begin{array}{cc}
d_2 & 1 \\[4pt]
-1 & 0
\end{array}
\biggr)
\biggl(
\begin{array}{cc}
d_3 & 1 \\[4pt]
-1 & 0
\end{array}
\biggr)
\biggl(
\begin{array}{cc}
d_4 & 1\\[4pt]
-1 & 0
\end{array}
\biggr)=
\biggl(
\begin{array}{cc}
-1 & 0\\[4pt]
0 & -1
\end{array}
\biggr).
$$
This relates the topic to two classical subjects:
the theory of continued fractions and the theory of linear difference equations.
The Gauss relations were the main motivation for
Coxeter~\cite{Cox} to develop the notion of frieze patterns,
relating projective geometry to combinatorics.
Friezes provide a special parametrization of $\cM_{0, N}$;
see~\cite{SVRS} and the appendix of~\cite{MGOT}.

We regard the relations between the symplectic cross-ratios
of $\cL_{n, N}$ as multi-dimensional analogs of the Gauss relations.
Building on preliminary versions of this article,
Morier-Genoud~\cite{Sop} has studied the combinatorial
aspects of $\cL_{2, N}(\C)$, the moduli space of Legendrian
$N$-gons in $\C\bP^3$.  Her work indicates that in general,
$\cL_{n,N}(\K)$ has a rich combinatorial structure related to friezes.

\subsection{Outline of results} \label{Outline}

It is natural to ask for a coordinate system on
the moduli space $\cL_{n, N}(\K)$ of $(n, N)$-\Lc s.
In this article we show that this question is vacuous when
$N$ is $2n$ or $2n+1$, and answer it when $N$ is $2n+2$.
Our coordinates are given by the \scr,
a direct analog of the classical cross-ratio: we show that
$\cL_{n, 2n+2}(\K)$ is parametrized by \scr s
and determine its structure as an algebraic variety.
We expect that \scr s parametrize $\cL_{n, N}(\K)$ for all $N$.
The exposition is organized as follows.

In Section~\ref{LCs} we define \scr s and show that they provide
continuous invariants on $(n, N)$-\Lc s for $N \ge 2n + 2$.
We also deduce the dimension of $\cL_{n, N}(\K)$
and define {\em opposite\/} configurations and equivalence classes,
which over $\R$ are distinguished by {\em sign invariants.\/}

Section~\ref{Main Geo Results} contains our main geometric results,
which we summarize here:

\begin{itemize}

\item
$(n, 2n)$-configurations are all generic and
equivalent over both $\C$ and $\R$.

\smallbreak \item
$(n, 2n+1)$-configurations are all generic.
Over $\C$ they are all equivalent, and
over $\R$ there are two equivalence classes,
which are opposite.

\smallbreak \item
$(n, 2n+2)$-configurations admit $n+1$ diametric \scr s $c_0, \ldots, c_n$:
\begin{equation} \label{L(n, 2n+2) CRs}
   c_i :=
   \frac
   {\om(x_i, x_{i+n+1})\, \om(x_{i+1}, x_{i+n+2})}
   {\om(x_i, x_{i+n+2})\, \om(x_{i+1}, x_{i+n+1})},
\end{equation}
the $x_i$ being arbitrary non-zero points on the lines of the configuration.

Over $\C$, generic $(n, 2n+2)$-configurations are equivalent
if and only if they have the same diametric cross-ratios.
Over $\R$, generic $(n, 2n+2)$-configurations with the same cross-ratios
are either equivalent or in opposite equivalence classes.

The moduli space $\cL_{n, 2n+2}(\K)$ is $n$-dimensional.
The $n+1$ cross-ratios satisfy the relation~(\ref{GenEq}),
and any collection of non-zero $\K$-scalars
$(c_0, \ldots, c_n)$ satisfying~(\ref{GenEq}) is the set
of cross-ratios of an $(n, 2n+2)$-configuration over $\K$.
Thus the cross-ratios are coordinates describing
$\cL_{n, 2n+2}(\K)$ as an algebraic hypersurface in $\K^{n+1}$.
These results are presented in Theorem~\ref{Main Thm}.

\end{itemize}
For $(n, 2n+2)$-configurations the proof has several components
and is given in Section~\ref{PMT Sec}.

In Section~\ref{Normalizations} we present
certain normalized choices of the $x_i$
generalizing Section~\ref{Hexagons}.
For $n$ even there is an essentially unique choice
such that the symplectic subdiameters
$\omega(x_i, x_{i+n})$ are all~$1$.
This normalization provides an alternate coordinate system on
$\cL_{n, 2n+2}(\K)$: the symplectic diameters $a_0, \ldots, a_n$,
where $a_i := \omega(x_i, x_{i+n+1})$.
These diameters are determined up to an overall choice of sign,
and for $\K = \R$ they are either all real or all pure imaginary.
In this coordinate system the relation~(\ref{GenEq})
can be written in terms of the celebrated classical determinants
called {\em continuants.\/}  This connection is emphasized
in Section~\ref{Even Complex Norms};
see Theorem~\ref{EquivThm}.

For $n$ odd, one cannot in general choose the points $x_i$
so that the symplectic subdiameters are all~$1$,
but one can choose them so that the subdiameters
alternate between a scalar $\mu$ and its reciprocal,
and the two alternating products of diameters are equal:
$a_0 a_2 \cdots a_{n-1} = a_1 a_3 \cdots a_n$.
Here $\mu$ is determined up to a $(2n+2)^{\ndup}$
root of unity, and the $a_i$ are determined
up to an overall $(n+1)^{\stup}$ root of unity.

An old idea of projective differential geometry
consists in representing geometric objects
such as curves or configurations of points via
differential or difference operators.
Following this approach, in Section~\ref{SLDEs}
we realize the moduli space of Lagrangian configurations
as the quotient by rescaling of the space of symmetric linear difference
equations with periodic coefficients and antiperiodic solutions;
see Theorem~\ref{MonThm v2}.

We conclude in Section~\ref{Remarks} with a preliminary discussion
of $\cL_{n, 2n+3}$, including a general result on normalizations
in the case that $N/\GCD(n, N)$ is odd, and relations on cross-ratios
for $\cL_{2, 7}$ and $\cL_{3, 9}$, the moduli spaces
of generic Legendrian heptagons in $\K\bP^3$ and
Legendrian nonagons in $\K\bP^5$.

\section{Lagrangian configurations and their moduli spaces} \label{LCs}

In this section we collect some basic properties of Lagrangian configurations
and the action of $\Sp(2n, \K)$ on them.
We prove that the action is free on generic configurations
and introduce two types of invariants:
continuous invariants known as symplectic cross-ratios,
and certain discrete sign invariants.

\subsection{Symplectic cross-ratios}

Consider two pairs of points in $(\K^{2n},\omega)$, $(x_1,x_2)$ and $(y_1,y_2)$,
such that $\omega(x_1, y_2)$ and $\omega(x_2, y_1)$ are non-zero.
We define their {\it \scr\/} to be
\begin{equation}
\label{CRaT}
[x_1, x_2; y_1, y_2] :=
\frac
{\om(x_1, y_1)\, \om(x_2, y_2)} {\om(x_1, y_2)\, \om(x_2,y_1)}.
\end{equation}

The \scr\ is obviously invariant with respect to both the action
of the symplectic group $\Sp(2n,\K)$ and rescalings
$x_i \mapsto \l_i x_i$ and $y_i \mapsto \mu_i y_i$.
Therefore it is in fact a symplectic invariant of 
two pairs of 1-dimensional subspaces in $\K^{2n}$,
or equivalently, of two pairs of points in $\K\bP^{2n-1}$.
Observe the symmetries
\begin{equation} \label{scr symmetries}
   [x_2, x_1; y_1, y_2] = [x_1, x_2; y_1, y_2]^{-1}, \qquad
   [y_1, y_2; x_1, x_2] = [x_1, x_2; y_1, y_2].
\end{equation}

\begin{rem}
For $n > 1$, (\ref{CRaT}) is not the only symplectic invariant of
a quadruple $(\K x_1, \K x_2, \K y_1, \K y_2)$ of lines in $\K^{2n}$.
However, it {\it is\/} the only such invariant if

\begin{itemize}
\smallbreak \item
$\langle x_1,x_2 \rangle$ and $\langle y_1,y_2 \rangle$ are
{\em isotropic,\/} i.e., $\om(x_1,x_2)$ and $\om(y_1,y_2)$ are $0$, and

\smallbreak \item
$(x_1,x_2,y_1,y_2)$ is {\it generic\/} under this condition.
\end{itemize}
\end{rem}

In the 1-dimensional case, (\ref{CRaT}) is nothing but
the classical cross-ratio of~$4$ points on the projective line.
In affine coordinates, it is given by the usual formula:
$$
\left[x_1,x_2;y_1,y_2\right]=
\frac
{\left(x_1 - y_1\right) \left(x_2 - y_2\right)}
{\left(x_1 - y_2\right) \left(x_2 - y_1\right)}.
$$
It is the unique $\PSL(2, \K)$-invariant of $(x_1, x_2, y_1, y_2)$.
Different partitions of the points into two pairs give six different
cross-ratios, but any one of them determines the others.

\begin{rem}
The cross-ratio plays a fundamental role in many areas,
from projective geometry to mathematical physics; for an overview, see~\cite{Lab}.
It is the discrete version of the Schwarzian derivative; see, e.g.,~\cite{OT}.
Different versions of multi-dimensional symplectic cross-ratios have been considered.
One is an invariant of a quadruples of Lagrangian planes related to the Maslov index;
again, see \cite{OT} for a survey.
Another is a unitary group invariant defined in the complex setting; see~\cite{FP}.
The symplectic cross-ratio~(\ref{CRaT}) is the most straightforward
generalization of the 1-dimensional cross-ratio.
It has been used to construct symplectic projective invariants in other settings;
see, e.g., \cite{OS} (p.~367) and~\cite{Yag}.
\end{rem}

\subsection{Gram matrices}

Given a collection $x_1, \ldots, x_m$ of vectors in $\K^{2n}$, we define
their {\em $\omega$-Gram matrix\/} $\Omega(x_1, \ldots, x_m)$
to be the $m \times m$ matrix whose entries are their symplectic products:
\begin{equation*}
   \Omega(x_1, \ldots, x_m)_{ij} := \omega(x_i, x_j).
\end{equation*}
We will use the following standard lemma throughout the paper.
Its proof is elementary and is omitted.

\begin{lem} \label{Gram}
Suppose that $x_1, \ldots, x_m$ and $x'_1, \ldots, x'_m$
are two collections of $m$ vectors in $\K^{2n}$.

\begin{enumerate}

\item[(i)]
$\Omega(x_1, \ldots, x_m)$ is of rank at most $2n$.

\smallbreak \item[(ii)]
$\Omega(x_1, \ldots, x_m)$ is of rank~$2n$ if and only if
$\langle x_1, \ldots, x_m \rangle = \K^{2n}$.

\smallbreak \item[(iii)]
If\/ $\Omega(x_1, \ldots, x_m)$ and $\Omega(x'_1, \ldots, x'_m)$
are equal and of rank~$2n$, then there is a unique symplectic
transformation $T$ such that $T(x_i) = x'_i$ for all~$i$.

\end{enumerate}
\end{lem}

\begin{lem} \label{ii from i}
Fix $N \ge 2n$ and let $(x_1, \ldots, x_N)$ be
an $N$-tuple of points in $\K^{2n}$.
Define $(x_i)_{i \in \Z}$ via $x_{i \pm N} := -x_i$, and assume that
$\langle x_{i+1}, \ldots, x_{i+n} \rangle$ is Lagrangian for all~$i$.
Then $(\K x_1, \ldots, \K x_N )$ is an $(n, N)$-\Lc\
if and only if $\omega(x_i, x_{i+n}) \not= 0$ for all~$i$.
\end{lem}

\begin{proof}
We must show that $\langle x_{i+1}, \ldots, x_{i+2n} \rangle = \K^{2n}$
for all~$i$ if and only if $\omega(x_i, x_{i+n}) \not= 0$ for all~$i$.
Consider the $\omega$-Gram matrix
$\Omega^i := \Omega(x_{i+1}, \ldots, x_{i+2n})$.
By Lemma~\ref{Gram}(ii), $\det(\Omega^i) \not= 0$ if and only if
$x_{i+1}, \ldots, x_{i+2n}$ form a basis of $\K^{2n}$.
The Lagrangian condition implies that the
block $2 \times 2$ form of $\Omega^i$ is
$\left(\begin{smallmatrix} 0 & A \\ -A^T & 0 \end{smallmatrix}\right)$,
where $A$ is upper triangular with diagonal entries
$\omega(x_{i+r}, x_{i+r+n})$, $1 \le r \le n$.
\end{proof}

\subsection{Continuous invariants}

Given an $(n, N)$-\Lc\ $(X_1, \ldots, X_N)$, fix {\em representatives:\/}
non-zero points $x_i$ on the lines $X_i$.
As in Section~\ref{Hexagons}, extend the $N$-tuple $(x_1, \ldots, x_N)$
to an $N$-antiperiodic sequence
\begin{equation} \label{antiperiodic}
   (x_i)_{i \in \Z}, \qquad x_{i \pm N} := -x_i.
\end{equation}
Write $\omega_{ij}$ for $\omega(x_i, x_j)$, and observe that
\begin{equation} \label{omega relations}
   \omega_{i+N, j} = \omega_{ji} = -\omega_{ij}, \qquad
   \omega_{j, i+N} = \omega_{ij}.
\end{equation}

Of course the $\omega_{ij}$ depend on the choice of the $x_i$,
but the \scr s of the configuration do not.
Define projective symplectic invariants
\begin{equation}
\label{Ccoord}
   c_{i_1 i_2 j_1 j_2} :=
   [x_{i_1}, x_{i_2}; x_{j_1}, x_{j_2}] =
   \frac{\omega_{i_1 j_1}\, \omega_{i_2 j_2}}
   {\omega_{i_1 j_2}\, \omega_{i_2 j_1}}.
\end{equation}
Note that $c_{i_1 i_2 j_1 j_2}$ has the symmetries~(\ref{scr symmetries})
and in addition is invariant under the addition of $N$ to any of the four indices.
Due to the Lagrangian condition
many of the $c_{i_1 i_2 j_1 j_2}$ vanish or are not well defined.
To be precise, define the
{\em $N$-cyclic distance\/} $|i-j|_N$
between any two integers~$i$ and~$j$ to be their
``separation modulo~$N$'':
\begin{equation*}
   |i-j|_N := \min \bigl\{|i-j+qN|: q \in \Z \bigr\}.
\end{equation*}

The Lagrangian condition is $\omega_{ij} = 0$ for $|i-j|_N < n$.
Therefore $c_{i_1 i_2 j_1 j_2}$ is either zero or undefined
unless $|i_{\e_i} - j_{\e_j}|_N \ge n$
for $\e_i$ and $\e_j$ either~$1$ or~$2$.
Moreover, if either $i_1 = i_2$ or $j_1 = j_2$,
then $c_{i_1 i_2 j_1 j_2}$ is~$1$ if defined.
It follows that there are no non-trivial cross-ratios
when $N$ is $2n$ or $2n+1$,
and the only cross-ratios of configurations with $N= 2n+2$
taking values other than~$0$ and~$1$
are the $c_i$ of~(\ref{L(n, 2n+2) CRs}):
\begin{equation} \label{L(n, 2n+2) CRs form 2}
   c_i = c_{i, i+1, i+n+1, i+n+2} =
   \frac{\omega_{i, i+n+1}\, \omega_{i+1, i+n+2}}
   {\omega_{i-n, i}\, \omega_{i+1, i+n+1}}.
\end{equation}
Observe that for $N = 2n+2$ these cross-ratios are $(n+1)$-periodic.
Lemma~\ref{ii from i} shows that in this case they are also always well defined.

We conjecture that in general, (\ref{Ccoord}) is a coordinate ring
on $\cL_{n, N}(\K)$ with polynomial relations.
This is confirmed only for $N \le 2n + 2$.
It would be interesting to find minimal cyclically invariant
subsets of~(\ref{Ccoord}) providing such coordinate rings.

\subsection{Opposite configurations and sign invariants} \label{OCSI}

Note that the negation of a symplectic form is also a symplectic form;
in particular, $-\omega$ is a symplectic form on $\K^{2n}$.
This leads to the notion of opposite \Lc s.
To be concrete, observe that the $2n \times 2n$ matrix
$Q = \left(\begin{smallmatrix} \Id & 0 \\ 0 & -\Id \end{smallmatrix}\right)$
has the property $\omega(Qx, Qy) = -\omega(x, y)$
for all $x$ and $y$ in $\K^{2n}$.

In light of the fact that conjugation by $Q$ preserves $\Sp(2n, \K)$,
the following lemma is clear.

\begin{lem} \label{opposite}

\begin{enumerate}

\item[(i)]
If $(X_1, \ldots, X_N)$ is a \Lc,
then so is $(Q X_1, \ldots, Q X_N)$.
We refer to them as\/ {\em opposites.\/}

\item[(ii)]
Opposite configurations have the same cross-ratios,
and if one is generic, then so is the other.

\item[(iii)]
If two \Lc s are equivalent, then their opposites are also equivalent.
\end{enumerate}
\end{lem}

Therefore we may speak of {\it opposite equivalence classes}
in $\cL_{n, N}(\K)$.
Because $Q$ is not symplectic,
a configuration is not {\it a priori} equivalent to its opposite.
In order to resolve the situation we define the {\em sign invariants.\/}
Let us write $\sgn$ for the sign function:
\begin{equation*}
   \sgn: \R \backslash \{0\} \to \{ \pm 1 \}.
\end{equation*}

Let $(\R x_1, \ldots, \R x_N)$ be a real $(n, N)$-\Lc, and
suppose that $j_0, j_1, \ldots, j_r$ are any integers
such that $j_0 \equiv j_r$ modulo $N$.
If $\omega_{j_{s-1}, j_s} \not= 0$ for all $s$, consider
the product $\prod _1^r \omega_{j_{s-1}, j_s}$.
Suppose we rescale each $x_j$ by some $\lambda_j$.
Because the $\lambda_j$ are by definition $N$-periodic,
$\lambda_{j_0} = \lambda_{j_r}$, so
the product rescales by the positive quantity
$\prod_1^r \lambda_{j_s}^2$.
This gives the following lemma.

\begin{lem} \label{sgn}
Let $(\R x_1, \ldots, \R x_N)$ be a real $(n, N)$-\Lc.
If $j_0, \ldots, j_r$ are integers such that $j_0 \equiv j_r$
modulo $N$ and $\omega_{j_{s-1}, j_s} \not= 0$ for all~$s$,
then $\sgn \bigl( \prod _1^r \omega_{j_{s-1}, j_s} \bigr)$
is an invariant of the configuration.
\end{lem}

The next proposition elucidates equivalence
and inequivalence of opposite configurations.
We write $\GCD$ for the greatest common divisor function.

\begin{prop} \label{OEI}

\begin{enumerate}

\smallbreak \item[(i)]
Opposite complex configurations are equivalent.

\smallbreak \item[(ii)]
Opposite real $(n, N)$-configurations are inequivalent if $N/ \GCD(n, N)$ is odd.

\smallbreak \item[(iii)]
Opposite real generic $(n, N)$-configurations are inequivalent for $N > 2n$.

\end{enumerate}
\end{prop}

\begin{proof}
Let $(\K x_1, \ldots \K x_N)$ be a configuration.
For~(i), note that $iQ$ is symplectic and $i Q \C x_j = Q \C x_j$.

For~(ii), note that the sign invariant
$\sgn \bigl( \prod_{s=1}^{N/\GCD(n, N)} \omega_{(s-1)n, sn} \bigr)$ negates under passage
to the representatives $Q x_j$ of the opposite configuration.

In the setting of~(iii), it is always possible to find a sequence $j_0, \ldots, j_r$
with $r$ odd, $j_0 \equiv j_r$ modulo $N$, and $|j_s - j_{s-1}|_N \ge n$,
whence the invariant $\sgn \bigl(\prod _1^r \omega_{j_{s-1}, j_s} \bigr)$
distinguishes between the configuration and its opposite.
\end{proof}

\subsection{Dimensions and standard configurations} \label{DSC}

We now discuss the dimension of $\cL_{n, N}$,
which determines the number of independent relations
which must be satisfied by any coordinate ring
of symplectic invariants of configurations.
It turns out that for $N \ge 2n+2$,
this dimension is always strictly less than
the number of non-trivial \scr s.
As stated in Section~\ref{Outline}, at $N = 2n+2$
the dimension is~$n$ and there are $n+1$ invariants $c_i$,
so there must be one relation.
It is given in Theorem~\ref{Main Thm}.

Let us begin by defining a convenient normal form
for sets of representatives of configurations.
Recall that we write $\{ e_1, \ldots, e_n, f_1, \ldots, f_n \}$
for the standard basis of $\K^{2n}$.

\begin{defn}
Given an $(n, N)$-\Lc\ $(\K x_1, \ldots, \K x_N)$,
the representatives $x_1, \ldots, x_N$ are said to be
{\em standard} if for some vectors
$g_i \in \langle f_1, \ldots, f_{i-1} \rangle$,
$x_1, \ldots, x_{2n}$ are given by
\begin{equation*}
   x_1 = e_1, \quad x_2 = e_2, \quad\ldots,\quad x_n = e_n, \qquad
   x_{n+1} = f_1, \quad x_{n+2} = f_2 + g_2,
   \quad\ldots,\quad x_{2n} = f_n + g_n.
\end{equation*}
\end{defn}

\begin{lem} \label{standard}

\begin{enumerate}

\smallbreak \item[(i)]
Every $(n, N)$-\Lc\ is symplectically equivalent
to a configuration with standard representatives.

\smallbreak \item[(ii)]
In a configuration with standard representatives,
$\omega(e_j, g_i) = 0$ for $j < i + 2n - N$.

\smallbreak \item[(iii)]
In a generic configuration with standard representatives,
$\omega(e_j, g_i) \not= 0$ for $i + 2n - N \le j < i$.

\end{enumerate}
\end{lem}

\begin{proof}
Given an $(n, N)$-configuration $(\K x_1, \ldots, \K x_N)$,
consider the Gram matrix $\Omega(x_1, \ldots, x_{2n})$.
Because it has the form noted in the proof of
Lemma~\ref{ii from i}, Lemma~\ref{Gram}(iii)
shows that there is a symplectic transformation
mapping $x_i$ to $e_i$ and $x_{n+i}$ to a multiple of $f_i + g_i$
for $1 \le i \le n$, $g_1$ being zero and
$g_2, \ldots, g_n$ having the desired form.
To complete the proof, rescale the representatives.
For~(ii) and~(iii), refer to the definitions of Lagrangian and generic.
\end{proof}

\begin{prop} \label{freely}
For $N > 2n$, the variety of $(n, N)$-\Lc s is $nN$-dimensional.
$\Sp(2n, \K)$ acts freely on generic configurations,
and so $\cL_{n, N}(\K)$ is $n (N - 2n-1)$-dimensional.
\end{prop}

\begin{proof}
Consider the process of constructing an $(n, N)$-configuration
by choosing first $X_1$, then $X_2$, and so on to $X_N$.
Count the number of degrees of freedom
available in choosing each $X_i$ as follows.
There are $2n-1$ degrees of freedom for $X_1$,
as it is simply an arbitrary line.
There are $2n-2$ degrees of freedom for $X_2$,
as $\langle X_1, X_2 \rangle$ must be isotropic,
and $2n-3$ for $X_3$, as $\langle X_1, X_2, X_3 \rangle$
must be isotropic.  Continuing, for $i \le n$
we find that there are $2n - i$ degrees for $X_i$.
For $n \le i \le N-n+1$ there are $n$ degrees for $X_i$,
as the only constraints arise from the requirement that
$\langle X_{i-n+1}, \ldots, X_i \rangle$ be isotropic.

There are $n-1$ degrees for $X_{N-n+2}$,
as in addition to the above isotropy requirement,
$\langle X_{N-n+2}, X_1 \rangle$ must be isotropic.
Continuing, for $N-n+1 \le i \le N$ there are $N+1-i$ degrees for $X_i$.
In particular, all of the $X_i$ can be chosen,
and there are a total of $nN$ degrees of freedom
in choosing the configuration.

To complete the proof, it suffices to prove that any symplectic
transformation $T$ stabilizing a generic standard configuration
$(\K x_1, \ldots, \K x_N)$ is the identitity $\Id$.
Keeping in mind that $T$ must stabilize $\K x_i$
but not necessarily $x_i$, we find that it must be of the form
$\left(\begin{smallmatrix} D & 0 \\ 0 & D^{-1} \end{smallmatrix}\right)$
for some diagonal matrix $D$.  Because $N > 2n$,
$\omega(e_{i-1}, g_i) \not= 0$, forcing $D$ to be scalar.
Finally, genericity implies that $\omega(e_n, x_{2n+1})$
and $\omega(f_1, x_{2n+1})$ are both non-zero.
Hence the condition that $T$ stabilize $\K x_{2n+1}$ forces $D = \Id$.
\end{proof}

\begin{cor}
$\cL_{n, N}(\R)$ and $\cL_{n, N}(\C)$ are smooth manifolds.
\end{cor}

\subsection{The case $N > 2n+2$}

We conclude this section with a few remarks on configurations with $N > 2n+2$.
We begin with a corollary of Lemma~\ref{ii from i}.  Let us single out the \scr s
\begin{equation} \label{gammas}
   \gamma_{ij} := c_{i, j+n, j, i+n}
   = \frac{\omega_{ij}\, \omega_{i+n, j+n}}
   {\omega_{i, i+n}\, \omega_{j, j+n}}.
\end{equation}

\begin{cor} \label{gamma ij}
Let $(\K x_1, \ldots, \K x_N)$ be an $(n, N)$-\Lc.
Then $\gamma_{ij}$ is defined for all~$i$ and~$j$,
and it is~$0$ for $|i-j|_N < n$ and~$1$ for $|i-j|_N = n$.
The configuration is generic if and only if
$\gamma_{ij}$ is non-zero whenever $|i-j|_N \ge n$.
\end{cor}

Observe that $\gamma_{i+N, j} = \gamma_{ji} = \gamma_{ij}$.
For $N = 2n+2$, all non-trivial \scr s are $\gamma_{ij}$'s, as
$c_i = \gamma_{i+1, i-n}$.
This is not true for $N > 2n+2$: for example, for $(2,7)$-configurations,
$c_{1, 2, 4, 5}$ is not a $\gamma_{ij}$.
However, we do have the relation
\begin{equation*}
   c_{i_1 i_2 j_1 j_2} c_{i_1+n, i_2+n, j_1+n, j_2+n} =
   \frac{\gamma_{i_1 j_1}\, \gamma_{i_2 j_2}}
   {\gamma_{i_1 j_2}\, \gamma_{i_2 j_1}}.
\end{equation*}

\section{The main results} \label{Main Geo Results}

In this section we describe the moduli space $\cL_{n, N}(\K)$
of generic $(n, N)$-configurations over $\K = \R$ or $\C$
for $N$ equal to $2n$, $2n+1$, and $2n+2$.
In the first two cases it is trivial and is described in Proposition~\ref{easy}.
Let us mention that a related result is proven
in~\cite{OS} for three lines in $\K^{2n}$.

The first non-trivial case, $\cL_{n, 2n+2}(\K)$,
is described in Theorem~\ref{Main Thm}:
there are $n+1$ cross-ratios, which satisfy a single relation.
In the case that $n$ is even and $\K = \C$, Theorem~\ref{EquivThm} provides
a combinatorial interpretation of this relation.

\subsection{The cases $N=2n$ and $N=2n+1$} \label{2n & 2n+1}

\begin{prop} \label{easy}

\begin{enumerate}
\item[(i)]
For $N = 2n$ or $2n+1$, all $(n, N)$-Lagrangian configurations are generic.

\smallbreak \item[(ii)]
$\cL_{n, 2n}(\K)$ is a single point for $\K = \R$ or $\C$.

\smallbreak \item[(iii)]
$\cL_{n, 2n+1}(\R)$ consists of two opposite points,
and $\cL_{n, 2n+1}(\C)$ is a single point.

\end{enumerate}
\end{prop}

\begin{proof}
Part~(i) is clear from Lemma~\ref{ii from i}
and the fact that $|i-j|_{2n+1}$ never exceeds~$n$.
For~(ii), Lemma~\ref{standard} shows that all
configurations are equivalent to
$(\K e_1, \ldots, \K e_n, \K f_1, \ldots, \K f_n)$.

For~(iii), Lemma~\ref{standard} shows that any configuration
is equivalent to some $(\K x_1, \ldots, \K x_{2n+1})$
with $x_i = e_i$ for $1 \le i \le n$, $x_{n+1} = f_1$,
and $x_{n+i} = f_i + b_i f_{i-1}$ for $2 \le i \le n$,
where the $b_i$ are non-zero scalars.
Define $b_1 := 1$, apply the diagonal symplectic transformation
$e_i \mapsto e _i \prod_1^i b_j^{-1}$ and $f_i \mapsto f_i \prod_1^i b_j$,
and rescale to deduce that we may assume $b_i =1$ for all~$i$.

Combine the Lagrangian condition with genericity
to see that $x_{2n+1}$ can be rescaled to take the form
$f_n + b_0 \sum_1^n (-1)^i e_i$ for some scalar $b_0 \not= 0$.
For $\K = \R$ and $b_0 > 0$ or $\K = \C$,
apply the symplectic transformation
$e_i \mapsto b_0^{1/2} e_i$ and $f_i \mapsto b_0^{-1/2} f_i$
and rescale to arrive at $x_{2n+1} = f_n + \sum_1^n (-1)^i e_i$.
For $\K = \R$ and $b_0 < 0$, use $(-b_0)^{1/2}$ in place of $b_0^{1/2}$
to arrive at $x_{2n+1} = f_n - \sum_1^n (-1)^i e_i$.
By Proposition~\ref{OEI}, the two signs of $b_0$
give opposite real equivalence classes.
\end{proof}

Combining this argument with Section~\ref{OCSI} yields the following corollary.
For $\e = \pm 1$, consider the $(n, 2n+1)$-configurations with representatives
\begin{equation} \label{2n+1 configs}
   \Bigl( e_1, e_2, \ldots, e_n, \quad
   \e f_1,\, \e (f_1 + f_2),\, \e (f_2 + f_3), \ldots, \e (f_{n-1} + f_n), \quad
   \e f_n + \sum_i^n (-1)^i e_i \Bigr).
\end{equation}
Observe that these representatives have $\omega_{i, n+i} = \e$ for all~$i$.

\begin{cor}

\begin{enumerate}
\item[(i)]
The unique element of $\cL_{n, 2n+1}(\C)$
is the class of~(\ref{2n+1 configs}) for $\e = 1$.

\smallbreak \item[(ii)]
The two opposite elements of $\cL_{n, 2n+1}(\R)$
are the classes of~(\ref{2n+1 configs}) for $\e = \pm 1$.

\smallbreak \item[(iii)]
An arbitrary $(n, 2n+1)$-\Lc\
$(\R x_1, \ldots, \R x_{2n+1})$ over $\R$
is equivalent to~(\ref{2n+1 configs}) with
$\e = \sgn \bigl( \prod_{i=0}^{2n} \omega_{i, n+i} \bigr)$.

\end{enumerate}
\end{cor}

\subsection{The first non-trivial case: $N=2n+2$} \label{2n+2}

We now give the description of $\cL_{n, 2n+2}(\K)$,
one of our main results.
Recall that the \scr s $c_i$ given in~(\ref{L(n, 2n+2) CRs})
and~(\ref{L(n, 2n+2) CRs form 2}) are $(n+1)$-periodic, are
the only non-trivial cross-ratios on $(n, 2n+2)$-configurations,
and are non-zero on generic configurations.

Let us use the standard notation $\lfloor x \rfloor$
for the integer part of a real number $x$.
In order to state the result, we define an index set
$\cI_n(r) \subset \{0, \ldots, n\}^r$
for $1 \le r \le \lfloor (n+1)/2 \rfloor$:
\begin{equation*}
   \cI_n(r) := \bigl\{ (i_1, \ldots, i_r):
   \mbox{\rm $0 \le i_1$, $i_s + 1 < i_{s+1}$ for $1 \le s < r$,
   $i_r \le n$, and if $i_1 = 0$, then $i_r < n$} \bigr\}.
\end{equation*}
Observe that this definition may be rephrased as follows:
$\cI_n(r)$ is the set of strictly increasing $r$-tuples
$0 \le i_1 < \cdots < i_r \le n$ such that the $(n+1)$-cyclic distances
$|i_s - i_{s'}|_{n+1}$ exceed~$1$ for all $s \not= s'$.
Let us give the initial cases explicitly:

\begin{itemize}

\smallbreak \item
For $n \ge 1$, $\cI_n(1) = \{0, 1, \ldots, n\}$.

\smallbreak \item
For $n \ge 3$, $\cI_n(2) = \{(0,2), (0,3), \ldots, (0, n-1);
(1,3), \ldots, (1, n); \ldots; (n-2, n) \}$.

\smallbreak \item
$\cI_5(3) = \{(0,2,4), (1,3,5)\}$, and
$\cI_6(3) = \{(0,2,4), (0,2,5), (0,3,5), (1,3,5), (1,3,6), (1,4,6)\}$.

\end{itemize}

\begin{thm}
\label{Main Thm}

\begin{enumerate}
\item[(i)]
On $\cL_{n, 2n+2}(\K)$, the $n+1$ \scr s
$c_0, \ldots, c_n$ satisfy the relation
\begin{equation} \label{GenEq}
0 = 1 + \sum_{r=1}^{\lfloor (n+1)/2 \rfloor}
\sum_{\cI_n(r)} \frac{(-1)^r}{c_{i_1} c_{i_2} \cdots c_{i_r}}.
\end{equation}

\smallbreak \item[(ii)]
$(c_0, \ldots, c_n)$ is a coordinate ring on $\cL_{n, 2n+2}(\K)$:
any two generic $(n, 2n+2)$-equivalence classes
with the same cross-ratios are equal for $\K = \C$,
and either equal or opposite for $\K = \R$.

\smallbreak \item[(iii)]
(\ref{GenEq}) is the only relation on the cross-ratios:
if $c_0, \ldots, c_n$ are arbitrary non-zero scalars in $\K$ satisfying~(\ref{GenEq}),
then they are the cross-ratios of some generic $(n, 2n+2)$-configuration over $\K$.

\end{enumerate}
\end{thm}

Thus $\cL_{n,2n+2}(\K)$ may be viewed as
a dense open subset of the algebraic hypersurface
in $\K^{n+1}$ defined by multiplying~(\ref{GenEq}) by $c_0 \cdots c_n$.
We prove Theorem~\ref{Main Thm} in Section~\ref{PMT Sec}.

\begin{exs}
\begin{enumerate}

\item[(a)]
For $\cL_{1,4}(\K)$, the relation on
the two cross-ratios of quadrilaterals in $\K \bP^1$
simply relates two forms of the classical cross-ratio:
$$ 
\frac{1}{c_0} + \frac{1}{c_1} = 1.
$$

\smallbreak \item[(b)]
For $\cL_{2, 6}(\K)$, the relation
on the three cross-ratios $c_0$, $c_1$, and $c_2$
of Legendrian hexagons in $\K \bP^3$ was given in~(\ref{FirstEq}).

\smallbreak \item[(c)]
The moduli space $\cL_{3, 8}(\K)$ of Legendrian octagons in $\K \bP^5$
is parametrized by the four cross-ratios $c_0$, $c_1$, $c_2$, and $c_3$,
subject to the relation
\begin{equation*}
   \frac{1}{c_0} + \frac{1}{c_1} + \frac{1}{c_2} + \frac{1}{c_3} -
   \frac{1}{c_0 c_2} - \frac{1}{c_1 c_3} = 1.
\end{equation*}

\smallbreak \item[(d)]
For $\cL_{4,10}(\K)$, the Legendrian decagons in $\K \bP^7$,
the five cross-ratios $c_0,\ldots,c_4$ satisfy
\begin{equation*}
   \frac{1}{c_0} + \frac{1}{c_1} + \frac{1}{c_2} + \frac{1}{c_3} + \frac{1}{c_4} -
   \frac{1}{c_0 c_2} - \frac{1}{c_1 c_3} - \frac{1}{c_2 c_4} -
   \frac{1}{c_3 c_0} - \frac{1}{c_4 c_1} = 1.
\end{equation*}
The following analog of Figure~1 depicts $c_0$ in this setting:
$$
 \xymatrix @!0 @R=0.32cm @C=0.45cm
 {
&&&x_0\ar@2{-}[dddddddd]&
\\
&x_9&&&& x_1\\
\\
x_8&&&&&&x_2\\
\\
x_7&&&&&&x_3\\
\\
&x_6\ar@{-}[rruuuuuuu]\ar@2{-}[rrrruuuuuu]&&&& x_4\\
&&&x_5\ar@{-}[rruuuuuuu]&
}
$$
\begin{center}
Figure 2.
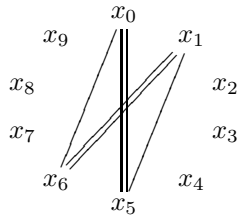
The cross-ratio $c_0 =
\frac
{\om_{0,5} \, \om_{1,6}}
{\om_{0,6} \, \om_{1,5}}$
on $\cL_{4,10}$.
\end{center}
\end{enumerate}
\end{exs}

\subsection{The cyclic continuant}
\label{Even Complex Norms}

In Section~\ref{Normalizations} we will describe certain
normalized choices of representatives of $(n, 2n+2)$-\Lc s.
The nature of the normalizations depends on
the field $\K$ and the parity of~$n$.  The situation is
simplest for $n$ even and $\K = \C$, the case generalizing
the hexagonal example in Section~\ref{Hexagons}.
At this point we state the relevant normalization,
Proposition~\ref{even complex norm},
along with the corresponding specialization of
Theorem~\ref{Main Thm}, Theorem~\ref{EquivThm}.
A notable feature of this specialization is that the
relation~(\ref{GenEq}) is expressed in terms of continuants.
The proofs will be given in Section~\ref{ECN proofs}.

The classical {\em continuant\/} is the tridiagonal determinant
\begin{equation} \label{ContEq}
K_n(a_1,\ldots,a_n):=
\left|
\begin{array}{cccccc}
a_1 & 1 &&& \\[4pt]
1 & a_2 & 1 && \\[4pt]
&\ddots&\ddots&\!\!\ddots& \\[4pt]
&& 1 & a_{n-1} &\!\!\!1 \\[4pt]
&&&\!\!\! 1 &\!\! a_{n}
\end{array}
\right|.
\end{equation}
This remarkable polynomial has a long history.
It arises in the theory of continued fractions and many other areas;
see for example \cite{CoOv} and references therein.

The main topic of \cite{CoOv} is the {\em cyclic continuant,\/}
defined as
\begin{equation} \label{rotundus}
   R_{n+1}(a_0, \ldots, a_n) :=
   K_{n+1}(a_0, \ldots, a_n) - K_{n-1}(a_1, \ldots, a_{n-1}).
\end{equation}
It too is related to continued fractions, via the identity
$$
 R_{n+1}(a_0, ,\ldots,a_n)=
\mathrm{tr}\biggl(
\begin{array}{cc}
a_n&-1\\[4pt]
1&0
\end{array}
\biggr)
\biggl(
\begin{array}{cc}
a_{n-1}&-1\\[4pt]
1&0
\end{array}
\biggr)
\cdots
\biggl(
\begin{array}{cc}
a_0&-1\\[4pt]
1&0
\end{array}
\biggr).
$$
As an illustration,
let us display $R_5(a_0, a_1, a_2, a_3, a_4)$:
\begin{equation*}
   a_0 a_1 a_2 a_3 a_4 -
   \bigl( a_0 a_1 a_2 + a_1 a_2 a_3 + a_2 a_3 a_4 +
   a_3 a_4 a_0 + a_4 a_0 a_1 \bigr) +
   \bigl( a_0 + a_1 + a_2 + a_3 + a_4 \bigr).
\end{equation*}

In the next two statements, let $(X_0, \ldots, X_{2n+1})$ be a generic
$(n, 2n+2)$-\Lc\ over $\K$ with representatives $x_0, \ldots, x_{2n+1}$,
symplectic products $\omega_{ij} := \omega(x_i, x_j)$, and cross-ratios
$c_0, \ldots, c_n$.
Recall from Section~\ref{Outline} that we refer to the $(2n+2)$-periodic sequence
$(\omega_{i, i+n})_i$ as the set of {\em symplectic subdiameters\/}
of $(x_i)_i$, and the $(n+1)$-periodic sequence $(\omega_{i, i+n+1})_i$
as the set of {\em symplectic diameters.\/}

\begin{prop} \label{even complex norm}
For $n$ even and\/ $\K = \C$, $(X_0, \ldots, X_{2n+1})$
admits exactly four choices of representatives
whose symplectic subdiameters are all\/~$1$.
Fix such a choice, $x_0, \ldots, x_{2n+1}$,
and set $a_i := \omega_{i, i+n+1}$.

\begin{enumerate}

\item[(i)]
The four choices are $(x_i)_i$, $(-x_i)_i$, $((-1)^i x_i)_i$, and $(-(-1)^i x_i)_i$.

\smallbreak \item[(ii)]
The first two choices in~(i) have symplectic diameters $(a_i)_i$,
and the second two have symplectic diameters $(-a_i)_i$.

\smallbreak \item[(iii)]
The cross-ratios of the configuration are $c_i = a_i a_{i+1}$.
The symplectic diameters satisfy
\begin{equation*}
   a_i^2 = \frac{c_i c_{i+2} \cdots c_{i+n}}{c_{i+1} c_{i+3} \cdots c_{i+n-1}}.
\end{equation*}

\end{enumerate}
\end{prop}

This result shows that the symplectic diameters
of the choices of representatives whose
symplectic subdiameters are all\/~$1$ are invariants
of the configuration, defined up to an overall sign.
We refer to them as the {\em normalized symplectic diameters\/}
and denote them by $\pm(a_0, \ldots, a_n)$.

In reading the next result, keep in mind that
$K_n$ and hence also $R_n$ are of parity $(-1)^n$
under negation of all their arguments, so
$R_{n+1}(a_0, \ldots, a_n)$ vanishes if and only if
$R_{n+1}(-a_0, \ldots, -a_n)$ vanishes.

\begin{thm} \label{EquivThm}
Let $n$ be even.

\begin{enumerate}
\item[(i)]
The cyclic continuant of the normalized symplectic diameters
on $\cL_{n, 2n+2}(\C)$ vanishes:
\begin{equation*}
   R_{n+1}(a_0, \ldots, a_n) = 0.
\end{equation*}

\smallbreak \item[(ii)]
Equivalence classes in $\cL_{n, 2n+2}(\C)$
with the same normalized symplectic diameters are equal.

\smallbreak \item[(iii)]
If\/ $\pm(a_0, \ldots, a_n)$ are arbitrary non-zero complex scalars
with cyclic continuant zero, then they are the normalized symplectic diameters
of some equivalence class in $\cL_{n, 2n+2}(\C)$.

\end{enumerate}
\end{thm}

As noted above, Proposition~\ref{even complex norm} and~Theorem~\ref{EquivThm} are proven in Section~\ref{ECN proofs}.

\section{Pfaffians and the proof of Theorem~\ref{Main Thm}}
\label{PMT Sec}

This section contains the proof of our main result, Theorem~\ref{Main Thm}.

\subsection{Tridiagonal determinants and the proof of Theorem~\ref{Main Thm}{\rm (i)}}
\label{MTi}

The following formula is well-known and easily proven by induction.

\begin{prop} \label{tridiag det}
The tridiagonal matrix
\begin{equation} \label{An}
   A_n := \left(
   \begin{array}{cccccc}
   a_{11} &\ \ \ a_{12} &&& \\[6pt]
   a_{21} &\ \ \ a_{22} & a_{23} && \\[6pt]
   &\!\!\ddots&\!\!\ddots&\!\!\ddots&\\[8pt]
   &&\ \ \ a_{n-1, n-2} &\ \ \ a_{n-1, n-1} & a_{n-1, n} \\[8pt]
   &&&\ \ \ a_{n, n-1} & a_{nn}
   \end{array}
   \right)
\end{equation}
has determinant
\begin{equation} \label{tridiag det formula}
   \det(A_n) = \Bigl( \prod_{k = 1}^n a_{kk} \Bigr)
   \sum_{r=0}^{\lfloor n/2 \rfloor} (-1)^r
   \sum_{\substack{1 \le i_1 < \cdots < i_r < n, \\
   i_s + 1 < i_{s+1}\, \forall\, 1 \le s < r}}
   \prod_{s=1}^r\ 
   \frac{a_{i_s, i_s+1}\, a_{i_s+1, i_s}}{a_{i_s i_s}\, a_{i_s+1, i_s+1}},
\end{equation}
where the summand at $r=0$ is understood to be~$1$.
\end{prop}

\begin{rem}
When the sub- and superdiagonal entries of $A_n$ are all~$1$,
its determinant is in fact the continuant $K_n$.
In this case Euler discovered a pleasing interpretation
of Proposition~\ref{tridiag det}, which generalizes as follows:
to write out all summands of~(\ref{tridiag det formula}),
start with $a_{11} \cdots a_{nn}$, and for each set of
disjoint adjacent pairs $(i, i+1)$, $1 \le i < n$,
replace $a_{ii} a_{i+1, i+1}$ by $-a_{i, i+1} a_{i+1, i}$.
We refer to this process as {\em Euler's replacement algorithm.\/}

\begin{exs}
Indicating pairs by parentheses,
$\det(A_3)$, $\det(A_4)$, and $\det(A_5)$ are, respectively,
\begin{align*}
   \bullet\ & a_{11} a_{22} a_{33}
   - (a_{12} a_{21}) a_{33} - a_{11} (a_{23} a_{32}), \\[4pt]
   \bullet\ & a_{11} a_{22} a_{33} a_{44}
   - (a_{12} a_{21}) a_{33} a_{44} -
   a_{11} (a_{23} a_{32}) a_{44} -
   a_{11} a_{22} (a_{34} a_{43})
   + (a_{12} a_{21}) (a_{34} a_{43}), \\[4pt]
   \bullet\ & a_{11} a_{22} a_{33} a_{44} a_{55}
   - (a_{12} a_{21}) a_{33} a_{44} a_{55} -
   a_{11} (a_{23} a_{32}) a_{44} a_{55} -
   a_{11} a_{22} (a_{34} a_{43}) a_{55} -
   a_{11} a_{22} a_{33} (a_{45} a_{54}) \\[4pt]
   & \hphantom{a_{11} a_{22} a_{33} a_{44} a_{55}}
   + (a_{12} a_{21}) (a_{34} a_{43}) a_{55}
   + (a_{12} a_{21}) a_{33} (a_{45} a_{54})
   + a_{11} (a_{23} a_{32}) (a_{45} a_{54}).
\end{align*}
\end{exs}
\end{rem}

Recall now that the {\em Pfaffian\/}
$\pf(S)$ is a polynomial in the entries of a
skew-symmetric matrix~$S$ whose square is $\det(S)$.
Consider the $(2n+2) \times (2n+2)$ matrix
\begin{equation} \label{Omega}
   \Omega_n := 
   \left(
    \begin{array}{cc} 
    \omega_{0, n} E & A\\ [8pt]
   -A^T & \omega_{n+1, 2n+1} E 
   \end{array} 
   \right),
\end{equation}
with ingredients defined as follows:
$E$ is the skew-symmetric $(n+1) \times (n+1)$ matrix
$e_{1,n+1} - e_{n+1,1}$ ($e_{ij}$ being the elementary matrix
whose $ij^{\thup}$ entry is~1 and whose other entries are~0),
$A$ is given by
\begin{equation} \label{omega tridiag}
   A = \left(
   \begin{array}{cccccc}
   \omega_{0, n+1} &\ \ \ \omega_{0, n+2} &&& \\[4pt]
   \omega_{1, n+1} &\ \ \ \omega_{1, n+2} & \omega_{1, n+3} && \\[4pt]
   &\ddots&\ddots&\!\!\ddots&\\[4pt]
   &&\ \ \ \omega_{n-1, 2n-1} &\ \ \ \omega_{n-1, 2n} & \omega_{n-1, 2n+1} \\[4pt]
   &&&\ \ \ \omega_{n, 2n} & \omega_{n, 2n+1}
   \end{array}
   \right),
\end{equation}
and the $\omega_{ij}$ are arbitrary scalars.
As a visual aid, we illustrate $\Omega_n$ in long form:
\begin{equation} \label{Omegan}
   \Omega_n=
\left(
\begin{array}{llllllllll}
&&& \hphantom{-} \omega_{0, n}
& \!\!\!\!\! \hphantom{-} \omega_{0, n+1}
& \!\!\!\!\! \!\!\!\!\! \omega_{0, n+2} & \\[4pt]
&&&& \!\!\!\!\! \hphantom{-} \omega_{1, n+1}
& \!\!\!\!\! \!\!\!\!\! \omega_{1, n+2}
& \omega_{1, n+3} &  \\
&&&&& \!\!\!\!\! \!\!\!\!\! \ddots & \!\!\!\!\! \ddots & \!\!\!\!\! \!\!\!\!\! \ddots \\
&&&&&& \!\!\!\!\! \ddots & \!\!\!\!\! \!\!\!\!\! \ddots
& \omega_{n-1, 2n+1} \\
-\omega_{0, n} &&&&&&
& \!\!\!\!\! \!\!\!\!\! \!\!\!\!\! \omega_{n, 2n}
& \omega_{n, 2n+1} \\[4pt]
-\omega_{0, n+1} & \!\!\! -\omega_{1, n+1} &&&&&&
& \omega_{n+1, 2n+1} \\
-\omega_{0, n+2} & \ddots & \ddots &&&&&& \\[4pt]
& \ddots & \ddots & \\
&&& \!\!\! -\omega_{n, 2n} &&&& \\[4pt]
&& \!\!\! -\omega_{n-1, 2n+1}
& \!\!\! -\omega_{n, 2n+1}
& \!\!\! -\omega_{n+1, 2n+1} &&
\end{array}
\right).
\end{equation}

\begin{prop} \label{pf Omega}
Recall $\cI_n(r)$ from Section~\ref{2n+2}.
The Pfaffian $\pf(\Omega_n)$ is given by the expression
\begin{equation} \label{pf Omega formula}
   (-1)^{n (n+1)/2}
   \Bigl( \prod_{k=0}^n \omega_{k, k+n+1} \Bigr)
   \sum_{r=0}^{\lfloor (n+1)/2 \rfloor} (-1)^r
   \sum_{\cI_n(r)}
   \prod_{s=1}^r \ \frac{\omega_{i_s, i_s+n+2}\, \omega_{i_s+1, i_s+n+1}}
   {\omega_{i_s, i_s+n+1}\, \omega_{i_s+1, i_s+n+2}},
\end{equation}
where the summand at $r=0$ is taken to be~$1$,
and whenever $\omega_{i, 2n+2}$ appears
it should be replaced by $\omega_{0, i}$.
\end{prop}

\begin{proof}
Given any $m \times m$ matrix $A$, let $A_{\mid}$ be the
$(m-2) \times (m-2)$ matrix obtained by deleting
the top and bottom rows and left and right columns of $A$.
The following lemma is proven in \cite{CoOv}.

\begin{lem} \label{pf det lemma}
For any $m \times m$ matrix $A$ and any scalars $x$ and $y$, one has
\begin{equation} \label{pf det formula}
   \pf \left( \begin{matrix} x E & A\\ -A^T & y E \end{matrix} \right)
   = (-1)^{m (m-1)/2} \bigl( \det(A) - xy \det(A_{\mid}) \bigr).
\end{equation}
\end{lem}

Applying both~(\ref{tridiag det formula}) and~(\ref{pf det formula})
to~(\ref{Omega}) yields~(\ref{pf Omega formula}).
\end{proof}

\begin{rem}
The Pfaffian (\ref{pf Omega formula})
may be interpreted via a cyclic version of
Euler's replacement algorithm,
the ``cyclic replacement algorithm'':
to write out all summands of
$(-1)^{n (n+1)/2}\, \pf(\Omega_n)$,
start with
\begin{equation*}
   \omega_{0, n+1} \omega_{1, n+2} \cdots \omega_{n, 2n+1},
\end{equation*}
and for each set of disjoint cyclically adjacent pairs $(i, i+1)$,
$0 \le i \le n$, replace $\omega_{i, i+n+1} \omega_{i+1, i+n+2}$
by $-\omega_{i, i+n+2} \omega_{i+1, i+n+1}$.
{\em Cyclically adjacent\/} indicates that the pair
$(n, n+1)$ is read as $(n, 0)$.  When it is in the set of pairs,
replace $\omega_{0, n+1} \omega_{n, 2n+1}$
by $-\omega_{0, n} \omega_{n+1, 2n+1}$.
To explain, note that by the replacement rule
$\omega_{i, 2n+2} = \omega_{0, i}$,
\begin{equation*}
   \omega_{0, n+1} \omega_{n, 2n+1} =
   \omega_{n, 2n+1} \omega_{n+1, 2n+2}, \qquad
   \omega_{n, 2n+2} \omega_{n+1, 2n+1} =
   \omega_{0,n} \omega_{n+1, 2n+1}.
\end{equation*}
\end{rem}

\begin{exs}
Let us give~(\ref{pf Omega formula}) explicitly for small $n$.
As in the examples below Proposition~\ref{tridiag det},
we indicate pairs with parentheses.
By~(\ref{pf det formula}) and~(\ref{omega tridiag}),
the summands from sets not containing the special pair $(n, 0)$
add up to an $(n+1) \times (n+1)$ tridiagonal determinant,
and the summands from sets containing it add up to
$\omega_{0, n} \omega_{n+1, 2n+1}$ times
an $(n-1) \times (n-1)$ tridiagonal determinant.
To emphasize this, we have separated
the two types of terms with square brackets
and factored $\omega_{0, n} \omega_{n+1, 2n+1}$
out of the second type.

Note that at $n = 1$ there are two ways to delete
the lone pair $(0, 1)$: as $(0, 1)$, or as $(1, 0)$.
For $n = 1$, $2$, and~$3$,
$(-1)^{n (n+1)/2}\, \pf(\Omega_n)$ is, respectively,
\begin{align*}
   \bullet\ & \Bigl[ \omega_{02} \omega_{13}
   - (\omega_{03} \omega_{12}) \Bigr]
   - (\omega_{01} \omega_{23}) \Bigl[ 1 \Bigr], \\[4pt]
   \bullet\ & \Bigl[ \omega_{03} \omega_{14} \omega_{25}
   - (\omega_{04} \omega_{13}) \omega_{25}
   - \omega_{03} (\omega_{15} \omega_{24}) \Bigr]
   - (\omega_{02} \omega_{35}) \Bigl[ \omega_{14} \Bigr], \\[4pt]
   \bullet\ & \Bigl[ \omega_{04} \omega_{15} \omega_{26} \omega_{37}
   - (\omega_{05} \omega_{14}) \omega_{26} \omega_{37}
   - \omega_{04} (\omega_{16} \omega_{25}) \omega_{37}
   - \omega_{04} \omega_{15} (\omega_{27} \omega_{36})
   + (\omega_{05} \omega_{14}) (\omega_{27} \omega_{36}) \Bigr] \\[4pt]
   & \hphantom{\Bigl[ \omega_{04} \omega_{15} \omega_{26} \omega_{37}}
   - (\omega_{03} \omega_{47})
   \Bigl[ \omega_{15} \omega_{26} - (\omega_{16} \omega_{25}) \Bigr].
\end{align*}
Observe that dividing by
$\omega_{0, n+1} \omega_{1, n+2} \cdots \omega_{n, 2n+1}$
leads to the examples in Section~\ref{2n+2},
and compare the terms in square brackets
to the examples below Proposition~\ref{tridiag det}.
\end{exs}

\medbreak \noindent {\bf Proof of~(\ref{GenEq}).}
Suppose now that $(\K x_0, \ldots, \K x_{2n+1})$ is an $(n, 2n+2)$-\Lc,
and revert to our customary notation $\omega_{ij} = \omega(x_i, x_j)$.
By the Lagrangian condition, the $\omega$-Gram matrix
$\Omega(x_0, \ldots, x_{2n+1})$ is precisely the matrix
$\Omega_n$ in~(\ref{Omega}), and by Lemma~\ref{Gram}(ii),
\begin{equation*}
   \det \bigl( \Omega(x_0, \ldots, x_{2n+1}) \bigr) = 0.
\end{equation*}
Hence Proposition~\ref{pf Omega} yields the following corollary.

\begin{cor} \label{Thm1i}
Let $(\K x_0, \ldots, \K x_{2n+1})$ be an $(n, 2n+2)$-\Lc.  Then
\begin{equation} \label{Thm1i formula}
   0 = \Bigl( \prod_{k=0}^n \omega_{k, k+n+1} \Bigr)
   \sum_{r=0}^{\lfloor (n+1)/2 \rfloor} (-1)^r
   \sum_{\cI_n(r)}
   \prod_{s=1}^r \ \frac{1}{c_{i_s}},
\end{equation}
where as usual, the summand at $r=0$ is understood to be~$1$.
\end{cor}

This in turn yields~(\ref{GenEq}) of Theorem~\ref{Main Thm}(i),
because generic $(n, 2n+2)$-configurations
have $\omega_{k, k+n+1} \not= 0$.
Note that~(\ref{pf Omega formula}) is polynomial
in the $\omega_{ij}$, so after cancellation,
Corollary~\ref{Thm1i} gives a non-trivial relation
even on non-generic configurations.

\subsection{Proof of Theorem~\ref{Main Thm}{\rm (ii)}}
\label{MTii}

Suppose that $(\K x_1, \ldots, \K x_{2n+2})$
and $(\K \tilde x_1, \ldots, \K \tilde x_{2n+2})$
are two generic $(n, 2n+2)$-\Lc s
with the same cross-ratios $c_1, \ldots, c_{n+1}$.
Following our convention $\omega_{ij} := \omega(x_i, x_j)$,
we set $\tilde \omega_{ij} := \tilde \omega(x_i, x_j)$.

By Lemma~\ref{Gram}(iii), in order to prove the two configurations
equivalent it suffices to find a renormalization
$x_i \mapsto \lambda_i x_i$ such that
$\lambda_i \lambda_j \omega_{ij} = \tilde \omega_{ij}$
for all $i$ and $j$.  By the Lagrangian condition,
we need only do this for $j = i+n$ and $i+n+1$.
The argument depends on the parity of $n$.

\medbreak \noindent {\bf The case of $n$ even.}
It is important to keep in mind that here $\GCD(n, 2n+2) = 2$.
Hence the subset of lines in an $(n, 2n+2)$-configuration
whose indices have a given parity may be written in either of the following ways:
\begin{equation*}
   \{\K x_i, \K x_{i+2}, \K x_{i+4}, \ldots, \K x_{i+2n}\} =
   \{\K x_i, \K x_{i+n}, \K x_{i+2n}, \ldots, \K x_{i+n^2}\}.
\end{equation*}

We begin with the case $\K = \C$.
For $i = 1, \ldots, 2n+2$, fix a square root $\chi_i$
of $\tilde \omega_{i, i+n} / \omega_{i, i+n}$.
Extend $(2n+2)$-periodically to define $\chi_i$
for $i \in \Z$.  Set
\begin{equation} \label{scale factors}
   \lambda_i \,:=\, \frac{\chi_i \chi_{i+2n} \chi_{i+4n} \cdots \chi_{i+n^2}}
   {\chi_{i+n} \chi_{i+3n} \cdots \chi_{i+(n-1)n}} \,=\,
   \prod_{r=0}^n \chi_{i+rn}^{(-1)^r}
\end{equation}
and check that $\lambda_i \lambda_{i+n} =
\tilde \omega_{i, i+n} / \omega_{i, i+n}$ for all~$i$.

Replacing $x_i$ by $\lambda_i x_i$, we may assume
$\omega_{i, i+n} = \tilde \omega_{i, i+n}$ for all~$i$.
Write $\rho_i$ for the ratio
$\tilde \omega_{i, i+n+1} / \omega_{i, i+n+1}$,
which is $(n+1)$-periodic.
Use the fact that the configurations have the same cross-ratios
to obtain $\rho_i \rho_{i+1} = 1$ for all $i$.  Hence
\begin{equation*}
   1 = \frac{(\rho_i \rho_{i+1}) (\rho_{i+2} \rho_{i+3})
   \cdots (\rho_{i+n} \rho_{i+n+1})}
   {(\rho_{i+1} \rho_{i+2}) (\rho_{i+3} \rho_{i+4})
   \cdots (\rho_{i+n-1} \rho_{i+n})} = \rho_i^2.
\end{equation*}
Deduce that the $\rho_i$ are either all~$1$ or all~$-1$.
In the former case we are done.
In the latter case, rescale again, replacing $x_i$ by $(-1)^i x_i$.
This leaves the $\omega_{i, i+n}$ unchanged
and negates the $\omega_{i, i+n+1}$, so again we are done.

\begin{ex}
For $n = 2$ and $n = 4$, the following diagrams depict the
equations giving the scale factors~(\ref{scale factors}) sending
$\omega_{i, i+n}$ to $\tilde\omega_{i, i+n}$ for $i$ even.
Those for $i$ odd are constructed independently.
$$
\xymatrix @!0 @R=0.55cm @C=1.1cm
{\\
&x_0\ar@<-4pt>@{<-}[lddd]\ar@<4pt>@{->}[rddd]&
\\
x_5&& x_1\\
\\
x_4\ar@{<-}[rr]&& x_2\\
&x_3&
}
\qquad \qquad \qquad \qquad
 \xymatrix @!0 @R=0.4cm @C=0.55cm
 {
&&&x_0\ar@<4pt>@{->}[rrddddddd]&
\\
&x_9&&&& x_1\\
\\
x_8\ar@{->}[rrrrrr]&&&&&&x_2\ar@{->}[llllldddd]\\
\\
x_7&&&&&&x_3\\
\\
&x_6\ar@<2pt>@{->}[rruuuuuuu]&&&& x_4\ar@{->}[llllluuuu]\\
&&&x_5&
}
$$
\begin{center}
Figure 3.
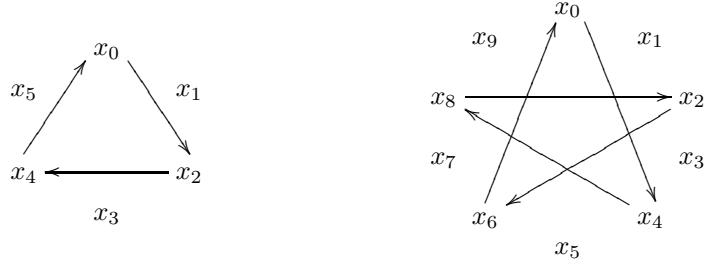
The rescaling scheme for $\cL_{2,6}$ and $\cL_{4,10}$.
\end{center}
\end{ex}

Now take $\K = \R$.  Because $(2n+2)/\GCD(n, 2n+2) = n+1$
is odd, the sign invariant
\begin{equation} \label{epsilon even}
   \sgn \Bigl(\prod_{r=0}^{n} \omega_{rn, (r+1)n} \Bigr)
   = \sgn \Bigl(\prod_{s=0}^{n} \omega_{2s, 2s+n} \Bigr)
\end{equation}
reverses under passage to the opposite configuration.
Therefore, replacing $(\R x_1, \ldots, \R x_{2n+2})$
by its opposite if necessary, we may assume that
\begin{equation} \label{even sign}
   \sgn \Bigl(\prod_{s=0}^{n} \omega_{2s, 2s+n} \Bigr)
   = \sgn \Bigl(\prod_{s=0}^{n} \tilde \omega_{2s, 2s+n} \Bigr).
\end{equation}

Under this assumption we will show that
the scale factors~(\ref{scale factors}) are all real,
so the two configurations are equivalent over $\R$.
This will complete the proof of
Theorem~\ref{Main Thm}(ii) for even~$n$.

To prove $\lambda_i$ real, we must prove $\lambda_i^2$ positive.
From~(\ref{scale factors}),
\begin{equation*}
   \lambda_i^2 = \prod_{r=0}^n
   \Bigl( \frac{\tilde \omega_{i+rn, i+(r+1)n}}
   {\omega_{i+rn, i+(r+1)n}} \Bigr)^{(-1)^r},
   \qquad
   \sgn(\lambda_i^2) = \sgn \Bigl( 
   \prod_{j \equiv i \mod 2}
   \tilde \omega_{j, j+n} /
   \omega_{j, j+n} \Bigr).
\end{equation*}
For $i$ even, this is positive by~(\ref{even sign}).
To prove it positive for $i$ odd, we must prove that
\begin{equation} \label{odd sign}
   \sgn \Bigl(\prod_{s=0}^{n} \omega_{2s+1, 2s+1+n} \Bigr)
   = \sgn \Bigl(\prod_{s=0}^{n} \tilde \omega_{2s+1, 2s+1+n} \Bigr).
\end{equation}
Check that
\begin{equation} \label{epsilon c}
   \prod_{i=0}^n c_i =
   \Bigl( \prod_{i=0}^n \omega_{i, i+n+1} \Bigr)^2
   \Bigl( \prod_{i=0}^{2n+1} \omega_{i, i+n} \Bigr)^{-1}.
\end{equation}
Because the two configurations have the same cross-ratios, we must have
\begin{equation*}
   \sgn \Bigl( \prod_{i=0}^{2n+1} \omega_{i, i+n} \Bigr) =
   \sgn \Bigl( \prod_{i=0}^{2n+1} \tilde \omega_{i, i+n} \Bigr).
\end{equation*}
Therefore~(\ref{even sign}) implies~(\ref{odd sign}).

\medbreak \noindent {\bf The case of $n$ odd.}
Here $\GCD(n, 2n+2) = 1$, so in contrast with the case that $n$ is even,
the entire set of lines in an $(n, 2n+2)$-configuration
may be listed with increments of~$n$:
\begin{equation*}
   \bigl\{ \K x_0, \K x_1, \K x_2, \ldots, \K x_{2n+1} \bigr\} =
   \bigl\{ \K x_0, \K x_n, \K x_{2n}, \ldots, \K x_{(2n+1)n} \bigr\}.
\end{equation*}

Set $\lambda_0 := 1$ and define $\lambda_n, \ldots, \lambda_{(2n+1)n}$
recursively by $\lambda_{rn} := \lambda_{(r-1)n}
\tilde\omega_{(r-1)n, rn} / \omega_{(r-1)n, rn}$
for $1 \le r \le 2n+1$.  This leads to
\begin{equation*}
   \lambda_{rn} = \prod_{s=1}^r
   \Bigl( \frac{\tilde\omega_{(s-1)n, sn}}
   {\omega_{(s-1)n, sn}} \Bigr)^{(-1)^{r-s}}.
\end{equation*}
At this point we have $\lambda_i \lambda_{i+n} =
\tilde \omega_{i, i+n} / \omega_{i, i+n}$
except possibly at $i \equiv -n$ modulo $2n+2$, where
\begin{equation*}
   \lambda_{-n} \lambda_0 = \prod_{s=1}^{2n+1}
   \Bigl( \frac{\tilde\omega_{(s-1)n, sn}}
   {\omega_{(s-1)n, sn}} \Bigr)^{(-1)^{s-1}}.
\end{equation*}
We claim that this is in fact
$\tilde\omega_{-n, 0} / \omega_{-n, 0}$.
The proof reduces to proving that the expression
\begin{equation*}
   \frac{\omega_{0, n} \omega_{2n, 3n} \cdots \omega_{2n^2, (2n+1)n}}
   {\omega_{n, 2n} \omega_{3n, 4n} \cdots \omega_{(2n+1)n, (2n+2)n}} =
   \frac{\omega_{0, n} \omega_{2, n+2} \cdots \omega_{2n, 3n}}
   {\omega_{1, n+1} \omega_{3, n+3} \cdots \omega_{2n+1, 3n+1}}
\end{equation*}
does not change if all the $\omega$'s are replaced with $\tilde\omega$'s.
This is true, because as the reader may check, it is equal to
\begin{equation*}
   \frac{c_0 c_2 \cdots c_{n-1}} {c_1 c_3 \cdots c_n},
\end{equation*}
and the $\omega$'s and $\tilde\omega$'s have the same cross-ratios.

Thus we may replace $x_i$ by $\lambda_i x_i$, giving
$\omega_{i, i+n} = \tilde\omega_{i, i+n}$ for all~$i$.
Then by equality of cross-ratios,
$\omega_{i, i+n+1} \omega_{i+1, i+n+2}$ is equal to
$\tilde\omega_{i, i+n+1} \tilde\omega_{i+1, i+n+2}$,
or, equivalently,
\begin{equation} \label{n odd scalar}
   \bigl(\tilde\omega_{i, i+n+1} / \omega_{i, i+n+1} \bigr)^{(-1)^i}
\end{equation}
is independent of~$i$.

If $\K = \C$, let $\chi$ be a square root of~(\ref{n odd scalar}).
If $\K = \R$, let $\chi$ be a square root of its magnitude.
Observe that for any $\delta$, the rescaling
$x_i \mapsto \delta^{(-1)^i} x_i$
leaves $\omega_{i, i+n}$ unchanged and multiplies
$\omega_{i, i+n+1}$ by $\delta^{2(-1)^i}$.
Therefore in the case $\K = \C$, replacing
the $x_i$ by $\chi^{(-1)^i} x_i$ gives
$\omega_{ij} = \tilde\omega_{ij}$, proving
the two configurations equivalent.
In the case $\K = \R$, the same argument
proves them equivalent when~(\ref{n odd scalar}) is positive.

In the case that $\K = \R$ and~(\ref{n odd scalar}) is negative,
this argument leaves us with $\omega_{i, i+n} = \tilde \omega_{i, i+n}$
and $\omega_{i, i+n+1} = -\tilde \omega_{i, i+n+1}$.
Replacing the $x_i$ by the additional rescaling $(-1)^i x_i$
then gives $\omega_{ij} = -\tilde\omega_{ij}$,
proving the configurations opposite.

\subsection{Proof of Theorem~\ref{Main Thm}{\rm (iii)}}
\label{MTiii}

Suppose that $c_0, \ldots, c_n$ are non-zero
scalars in $\K$ satisfying~(\ref{GenEq}), and extend them
$(n+1)$-periodically to $(c_i)_{i \in \Z}$.
We will construct an $(n, 2n+2)$-\Lc\
$(\K x_0, \K x_1, \ldots, \K x_{2n+1})$
having the given scalars $c_i$ as cross-ratios.

As an intermediate step, we construct from the $c_i$ scalars
$\omega_{ij}$ that will be equal to $\omega(x_i, x_j)$.
By~(\ref{omega relations}) and the Lagrangian condition,
it is only necessary to construct a $(2n+2)$-periodic sequence
$\omega_{i, i+n}$ and an $(n+1)$-periodic sequence $\omega_{i, i+n+1}$
such that the $c_i$ are given by~(\ref{L(n, 2n+2) CRs form 2}):
then $\omega_{i, i+n+2} = \omega_{i-n, i}$,
and the remaining $\omega_{ij}$ are~$0$.

\medbreak \noindent {\bf The $\omega_{ij}$ for $n$ even.}
Over $\C$, set $\omega_{i, i+n} := 1$ for all~$i$ and fix an $(n+1)$-periodic
sequence $\sigma_i$ of square roots of the $c_i$: $\sigma_i^2 = c_i$.
It is then simple to check that~(\ref{L(n, 2n+2) CRs form 2})
is satisfied if we set
\begin{equation} \label{a_i's}
   \omega_{i, i+n+1} :=
   \frac{\sigma_i \sigma_{i+2} \cdots \sigma_{i+n}}
   {\sigma_{i+1} \sigma_{i+3} \cdots \sigma_{i+n-1}}.
\end{equation}

Over $\R$, the same process works if $c_0 c_1 \cdots c_n$
is positive: the individual $\sigma_i$ may not be real, but
the $\omega_{i, i+n+1}$ are because their squares are positive.

If $c_0 c_1 \cdots c_n$ is negative, it suffices to modify the construction
as follows: set $\omega_{i, i+n} := (-1)^i$, let $\sigma_i$ be an
$(n+1)$-periodic sequence of square roots of the $-c_i$, and
again define the $\omega_{i, i+n+1}$ by~(\ref{a_i's}).

\medbreak \noindent {\bf The $\omega_{ij}$ for $n$ odd.}
Observe that in this case, (\ref{L(n, 2n+2) CRs form 2}) implies
$\prod_{i=0}^{2n+1} \omega_{i, i+n}^{(-1)^i} = \prod_0^n c_i^{(-1)^i}$.

We begin with an asymmetric choice
of the $\omega_{ij}$ that works over any field.
Define $\omega_{0, n} := \prod_0^n c_i^{(-1)^i}$, and
for $1 \le i \le 2n+1$, set $\omega_{i, i+n} := 1$.
Check that it suffices to set $\omega_{0, n+1} := 1$, $\omega_{1, n+2} := c_0$,
and in general, for $0 \le k \le \half (n-1)$,
\begin{equation*}
   \omega_{2k, 2k+n+1} := \frac{c_1 c_3 \cdots c_{2k-1}}{c_0 c_2 \cdots c_{2k-2}},
   \qquad
   \omega_{2k+1, 2k+n+2} := \frac{c_0 c_2 \cdots c_{2k}}{c_1 c_3 \cdots c_{2k-1}}.
\end{equation*}

Over $\C$ it is possible to choose the $\omega_{ij}$ more symmetrically,
as described in Section~\ref{Outline}.
Fix an $(n+1)$-periodic sequence $\eta_i$ such that $\eta_i^{2n+2} = c_i$.
Define the $\omega_{i, i+n}$ by
\begin{equation*}
   \mu := \frac{\eta_0 \eta_2 \cdots \eta_{n-1}}{\eta_1 \eta_3 \cdots \eta_n},
   \qquad \omega_{i, i+n} := \mu^{(-1)^i}.
\end{equation*}
The reader may check that then~(\ref{L(n, 2n+2) CRs form 2}) is satisfied by
fixing $\omega_{0, n+1}$ arbitrarily and setting
\begin{equation*}
   \omega_{2k, 2k+n+1} :=
   \mu^{4k} \omega_{0, n+1}
   \frac{c_1 c_3 \cdots c_{2k-1}}{c_0 c_2 \cdots c_{2k-2}},
   \qquad
   \omega_{2k+1, 2k+n+2} :=
   \mu^{-4k-2} \omega_{0, n+1}^{-1}
   \frac{c_0 c_2 \cdots c_{2k}}{c_1 c_3 \cdots c_{2k-1}}
\end{equation*}
for $0 \le k \le \half (n-1)$.
Requiring $a_0 a_2 \cdots a_{n-1} = a_1 a_3 \cdots a_n$
leads to the most symmetric choice:
\begin{equation*}
   \omega_{i, i+n+1} \,:=\,
   \frac{\eta_i^n \eta_{i+2}^{n-4} \cdots \eta_{i+n-1}^{-n+2}}
   {\eta_{i+1}^{n-2} \eta_{i+3}^{n-6} \cdots \eta_{i+n}^{-n}}
   \,=\, \prod_{j=0}^n \eta_{i+j}^{(-1)^j(n-2j)}.
\end{equation*}

\medbreak \noindent {\bf The $x_i$ for $n$ arbitrary.}
Suppose now that scalars $\omega_{ij}$ in $\K$ have been
chosen so as to satisfy~(\ref{L(n, 2n+2) CRs form 2}).
From these scalars we will construct an $(n, 2n+2)$-\Lc\
$(\K x_0, \ldots, \K x_{2n+1})$ over $\K$ which satisfies
$\omega(x_i, x_j) = \omega_{ij}$, and therefore has
the given cross-ratios $c_0, \ldots, c_n$.

The representatives $x_i$ are almost standard:
we set $x_i := e_i$ for $1 \le i \le n$, and
\begin{align*}
   & x_{1+n} := \omega_{1, 1+n} f_1, \qquad
      x_{2+n} := \omega_{1, 2+n} f_1 + \omega_{2, 2+n} f_2, \\[4pt]
   & x_{i+n} := \omega_{i-2, i+n} f_{i-2} +
   \omega_{i-1, i+n} f_{i-1} + \omega_{i, i+n} f_i,
   \quad 3 \le i \le n.
\end{align*}
It is immediate that with these definitions,
$\omega(x_i, x_j) = \omega_{ij}$ for $1 \le i, j \le 2n$.

The requirement that $\omega(x_i, x_j) = \omega_{ij}$
for $1 \le i \le 2n$ and $j = 2n+1$ or $2n+2$
now determines $x_{2n+1}$ and $x_0 = -x_{2n+2}$.
We find that
\begin{equation*}
   x_{2n+1} = \omega_{n-1, 2n+1} f_{n-1}
   + \omega_{n, 2n+1} f_n + \omega_{1+n, 2n+1} \sum_{i=1}^n
   (-1)^i \Bigl( \prod_{j=1}^i \omega_{j, j+n}^{-1} \Bigr) d_i e_i,
\end{equation*}
where $d_1 = 1$, $d_2 = \omega_{1, 2+n}$,
and the $d_i$ with $3 \le i \le n$ satisfy the recursion relation
\begin{equation} \label{det recursion}
   d_i = \omega_{i-1, i+n} d_{i-1} -
   \omega_{i-1, i+n-1} \omega_{i-2, i+n} d_{i-2}.
\end{equation}
In the same way we obtain
\begin{equation*}
   x_0 = -\omega_{0, n} f_n -
   \sum_{i=1}^n (-1)^i \Bigl( \prod_{j=1}^i \omega_{j, j+n}^{-1} \Bigr) d'_i e_i,
\end{equation*}
where $d'_1 = \omega_{0, 1+n}$, $d'_2 =
\omega_{0, 1+n} \omega_{1, 2+n} - \omega_{0, 2+n} \omega_{1, 1+n}$,
and the $d'_i$ with $3 \le i \le n$ also satisfy~(\ref{det recursion}).

In order to clarify~(\ref{det recursion}),
consider for any integers $0 \le j \le i \le n$
the following truncation of the tridiagonal
matrix $A$ in~(\ref{omega tridiag}):
\begin{equation*}
   A_{j, i} = \left(
   \begin{array}{cccccc}
   \omega_{j, j+n+1} &\ \ \ \omega_{j, j+n+2} &&& \\[4pt]
   \omega_{j+1, j+n+1} &\ \ \ \omega_{j+1, j+n+2} & \omega_{j+1, j+n+3} && \\[4pt]
   &\ddots&\ddots&\!\!\ddots&\\[4pt]
   &&\ \ \ \omega_{i-1, i+n-1} &\ \ \ \omega_{i-1, i+n} & \omega_{i-1, i+n+1} \\[4pt]
   &&&\ \ \ \omega_{i, i+n} & \omega_{i, i+n+1}
   \end{array}
   \right).
\end{equation*}
Let us write $\Delta_{j, i}$ for $\det(A_{j, i})$, and adopt the convention
$\Delta_{j, j-1} := 1$ and $\Delta_{j, j-2} := 0$.  It is clear that
the $\Delta_{j, i}$ satisfy a shifted version of~(\ref{det recursion}):
\begin{equation*}
   \Delta_{j, i} = \omega_{i, i+n+1} \Delta_{j, i-1} -
   \omega_{i, i+n} \omega_{i-1, i+n+1} \Delta_{j, i-2}.
\end{equation*}
Therefore $d_i = \Delta_{1, i-1}$ and
$d'_i = \Delta_{0, i-1}$ for $1 \le i \le n$.

The only remaining condition is $\omega(x_0, x_{2n+1}) = 0$.
After some simplification it reduces to
\begin{equation} \label{omega 0, 2n+1}
   \Delta_{0, n} - \omega_{0, n} \omega_{n+1, 2n+1} \Delta_{1, n-1} = 0.
\end{equation}
Recall the matrices $\Omega_n$ and $A$ from~(\ref{Omega}).
Because $A$ is $A_{0, n}$,
Proposition~\ref{pf Omega} and (\ref{pf det formula})
show that~(\ref{omega 0, 2n+1}) is equivalent to~(\ref{Thm1i formula}).
Because we assumed that the given $c_i$
satisfy~(\ref{GenEq}), these conditions hold.
This completes the proof of Theorem~\ref{Main Thm}.

\section{Normalized configurations}
\label{Normalizations}

As noted in Section~\ref{Even Complex Norms},
in this section we describe certain normalized choices
of representatives of $(n, 2n+2)$-configurations.
As in that section, let $(X_0, \ldots, X_{2n+1})$ be a generic
$(n, 2n+2)$-\Lc\ over $\K$ with representatives $x_0, \ldots, x_{2n+1}$,
symplectic products $\omega_{ij} := \omega(x_i, x_j)$, and cross-ratios
$c_0, \ldots, c_n$.

\subsection{The case of $n$ even and $\K = \C$}
\label{ECN proofs}

The results in this case were stated in Section~\ref{Even Complex Norms}.
Here we give their proofs.

\medbreak \noindent {\it Proof of Proposition~\ref{even complex norm}.\/}
By the constructions in Section~\ref{MTiii},
there exists a generic configuration
$(\C \tilde x_0, \ldots, \C \tilde x_{2n+1})$
with $\tilde\omega_{i, i+n} = 1$ and cross-ratios $c_0, \ldots, c_n$.
By Theorem~\ref{Main Thm}(ii), it is equivalent to
$(X_0, \ldots, X_{2n+1})$.  The images $x_0, \ldots, x_{2n+1}$
of $\tilde x_0, \ldots, \tilde x_{2n+1}$ under the equivalence
have $\omega_{i, i+n} = 1$ for all~$i$.

To see that there are exactly four such choices of $(x_i)_i$,
suppose that $(\lambda_i x_i)_i$ is another.
Observe that then $\lambda_{i+n} = \lambda_i^{-1}$,
so $\lambda_{i+rn} = \lambda_i^{(-1)^r}$.
But $\lambda_{i+n(n+1)} = \lambda_i$ by periodicity,
so $\lambda_i = \pm1$, whence $\lambda_{i+n} = \lambda_i$,
and $\lambda_i = \lambda_j$ for $i \equiv j$ modulo~2.
This proves~(i).  The remaining statements are immediate.
\hfill $\Box$

\medbreak\noindent {\it Proof of Theorem~\ref{EquivThm}.\/}
Part~(i) follows from Section~\ref{MTi}: by Lemma~\ref{pf det lemma},
for normalized representatives $(x_0, \ldots, x_{2n+1})$
the Pfaffian of the $\omega$-Gram matrix is, up to a sign,
$R_{n+1}(a_0, \ldots, a_n)$.
Parts~(ii) and~(iii) follow from Theorem~\ref{Main Thm}(ii) and~(iii)
and Proposition~\ref{even complex norm}(iii):
the cross-ratios determine $\pm(a_0, \ldots, a_n)$ and vice versa.
\hfill $\Box$

\subsection{The case of $n$ even and $\K = \R$}
\label{Even Real Norms}

Recall from Section~\ref{MTii} the sign invariants
\begin{equation*}
   \e_0 := \sgn \Bigl(\prod_{s=0}^{n} \omega_{2s, 2s+n} \Bigr), \qquad
   \e_1 := \sgn \Bigl(\prod_{s=0}^{n} \omega_{2s+1, 2s+1+n} \Bigr), \qquad
   \e_c := \sgn \Bigl(\prod_{i=0}^{n} c_i \Bigr)
\end{equation*}
of the configuration $(X_0, \ldots, X_{2n+1})$:
$\e_0$ is~(\ref{epsilon even}), $\e_1$ is the left side of~(\ref{odd sign}),
and by~(\ref{epsilon c}), $\e_c = \e_0 \e_1$.

\begin{prop} \label{even real norm}
For $n$ even and $\K = \R$, $(X_0, \ldots, X_{2n+1})$ admits exactly
four choices of representatives whose symplectic subdiameters are
$\omega_{i, i+n} = \e_{i \mod 2}$.  Fix such a choice, $x_0, \ldots, x_{2n+1}$,
and denote its symplectic diameters $\omega_{i, i+n+1}$ by $a^\R_i$.

\begin{enumerate}

\item[(i)]
In terms of $(x_i)_i$, the four choices are
as in Proposition~\ref{even complex norm}(i).
The first two have symplectic diameters $(a^\R_i)_i$,
and the second two have symplectic diameters $(-a^\R_i)_i$.
In particular, $\pm (a^\R_i)_i$ is an invariant of the configuration,
the collection of its {\em normalized real symplectic diameters.\/}

\smallbreak \item[(ii)]
If\/ $\e_c = 1$, then the normalized symplectic diameters
coincide with the normalized real symplectic diameters:
$\pm (a_i)_i = \pm (a^\R_i)_i$.

\smallbreak \item[(iii)]
If\/ $\e_c = -1$, then $\pm (a_i)_i = \pm \sqrt{-1}\, (a^\R_i)_i$.

\end{enumerate}
\end{prop}

\begin{proof}
By the constructions in Section~\ref{MTiii},
there exists a generic configuration
$(\R \tilde x_0, \ldots, \R \tilde x_{2n+1})$ with cross-ratios $c_i$,
such that if $\e_c = 1$, then $\tilde \omega_{i, i+n} = 1$ for all~$i$,
and if $\e_c = -1$, then $\tilde \omega_{i, i+n} = (-1)^i$ for all~$i$.
By Theorem~\ref{Main Thm}(ii), this configuration is equivalent
either to $(X_0, \ldots, X_{2n+1})$ or its opposite, and so,
replacing $(\R \tilde x_0, \ldots, \R \tilde x_{2n+1})$
by its opposite if necessary, we may assume that it is
equivalent to $(X_0, \ldots, X_{2n+1})$.
Then, recalling that passage to opposites negates symplectic products,
we find that the images $x_0, \ldots, x_{2n+1}$
of $\tilde x_0, \ldots, \tilde x_{2n+1}$ under the equivalence
have $\omega_{i, i+n} = \e_{i \mod 2}$.

The proof of~(i) goes exactly as in Proposition~\ref{even complex norm}.
For~(ii) and~(iii), fix a choice of $\sqrt{-1}$.
In~(ii), if $\e_0$ and $\e_1$ are both~$-1$, then the representatives
of the complex normalization of Section~\ref{Even Complex Norms}
may be taken to be $\sqrt{-1}\, (x_i)_i$, while if they are both $1$,
then the real and complex normalizations coincide.

In~(iii), if $\e_0 = 1$ and $\e_1 = -1$,
the $\C$-normalized representatives may be taken to be
$x_i$ for $i$ even and $\sqrt{-1}\, x_i$ for $i$ odd, while if
$\e_0 = -1$ and $\e_1 = 1$, they may be taken to be
$\sqrt{-1}\, x_i$ for $i$ even and $x_i$ for $i$ odd.
To summarize, in all cases the $\C$-normalized representatives
are $\bigl( \sqrt{-1}^{(1-\e_{i \mod 2})/2} x_i \bigr)_i$.
The relation between $\pm (a_i)_i$ and $\pm (a^\R_i)_i$
now follows easily.
\end{proof}

\subsection{The case of $n$ odd and $\K = \C$}
\label{Odd Complex Norms}

Recall from Section~\ref{MTiii} that here
$\prod_{i=0}^{2n+1} \omega_{i, i+n}^{(-1)^i} = \prod_0^n c_i^{(-1)^i}$.

\begin{prop} \label{odd complex norm}
For $n$ odd, $\K = \C$, and $\mu$ any $(2n+2)^{\ndup}$ root
of\/ $\prod_0^n c_i^{(-1)^i}$, $(X_0, \ldots, X_{2n+1})$
admits exactly $(2n+2)$ choices of representatives
such that $\omega_{i, i+n} = \mu^{(-1)^i}$ for all~$i$
and\/ $\prod_0^n \omega_{i, i+n+1}^{(-1)^i} = 1$.
Fix such a choice, $x_0, \ldots, x_{2n+1}$,
and set $a_i := \omega_{i, i+n+1}$.

\begin{enumerate}

\item[(i)]
The $2n+2$ choices are $(\delta^{(-1)^i} x_i)_i$, where $\delta$
runs over the $(2n+2)^{\ndup}$ roots of unity.

\smallbreak \item[(ii)]
The symplectic diameters corresponding to
any given choice of $\delta$ are $(\delta^{2 (-1)^i} a_i)_i$.

\smallbreak \item[(iii)]
The cross-ratios of the configuration are $c_i = \mu^{2 (-1)^i} a_i a_{i+1}$.

\end{enumerate}
\end{prop}

\begin{proof}
The discussion in Section~\ref{MTiii} shows that for any
$(2n+2)^{\ndup}$ root $\mu$ of $\prod_0^n c_i^{(-1)^i}$,
there exists a configuration $(\C \tilde x_0, \ldots, \C \tilde x_{2n+1})$ 
with $\tilde\omega_{i, i+n} = \mu^{(-1)^i}$ and cross-ratios $c_0, \ldots, c_n$.
By Theorem~\ref{Main Thm}(ii), it is equivalent to
$(X_0, \ldots, X_{2n+1})$.  The images $x_0, \ldots, x_{2n+1}$
of $\tilde x_0, \ldots, \tilde x_{2n+1}$ under the equivalence
have $\omega_{i, i+n} = \mu^{(-1)^i}$ for all~$i$.

For~(i), first check that if a renormalization $x_i \mapsto \lambda_i x_i$
preserves $\omega_{i, i+n}$ for all $i$, then it is of the form
$\lambda_i = \delta^{(-1)^i}$ for some~$\delta$.
Then check that $\prod_0^n \omega_{i, i+n+1}^{(-1)^i} = 1$
if and only if $\delta^{2n+2} = 1$.  The remaining statements are clear.
\end{proof}

We remark that the general rescaling $x_i \mapsto \lambda_i x_i$
going between normalizations as above with
different choices of $\mu$ is $\lambda_i := \xi^{(-1)^i (ip + q)}$, where
$\xi$ is any primitive $(2n+2)^{\ndup}$ root of unity and
$p$ and $q$ are arbitrary elements of $\Z_{2n+2}$.
It transforms $\omega_{i, i+n}$ from $\mu^{(-1)^i}$ to
$(\xi^{-pn} \mu)^{(-1)^i}$ and $\omega_{i, i+n+1}$ from $a_i$
to $(-1)^p \xi^{2(-1)^i (ip+q)} a_i$, i.e.,
\begin{equation*}
   \bigl( \mu, a_i \bigr)_i \mapsto
   \bigl( \xi^{-pn} \mu, (-1)^p\, \xi^{2(-1)^i (ip+q)} a_i \bigr)_i.
\end{equation*}

We will not formally state the specialization of Theorem~\ref{Main Thm}
corresponding to the normalization in Proposition~\ref{odd complex norm},
but let us describe the specialization of the relation~(\ref{GenEq}):
it becomes the vanishing of the quantity obtained from the product
$a_0 \cdots a_n$ by applying the ``$\mu$-cyclic replacement rule'':
replace cyclically adjacent pairs $a_i a_{i+1}$ by $-\mu^{-2(-1)^i}$.
For example, at $n=1$ and~$3$,
\begin{align*}
   & 0 = a_0 a_1 - \mu^{-2} - \mu^2, \\[4pt]
   & 0 = a_0 a_1 a_2 a_3 - \mu^{-2} (a_0 a_1 + a_2 a_3)
   + \mu^{-4} + \mu^4.
\end{align*}

\subsection{The case of $n$ odd and $\K = \R$}
\label{Odd Real Norms}

Here we have only found natural normalizations under certain positivity conditions.

\begin{prop} \label{odd real norm}
Suppose that $n$ is odd, $\K = \R$, and $\prod_0^n c_i^{(-1)^i}$
is positive, and let $\mu$ be its positive $(2n+2)^{\ndup}$ root.
Then $(X_0, \ldots, X_{2n+1})$ admits choices of representatives
such that $\omega_{i, i+n} = \mu^{(-1)^i}$ for all~$i$.

If both $c_0 c_2 \cdots c_{n+1}$ and $c_1 c_3 \cdots c_n$
are positive, then exactly two such choices satisfy in addition
$\prod_0^n \omega_{i, i+n+1}^{(-1)^i} = 1$.
Otherwise there is no such choice.
\end{prop}

\begin{proof}
Given any representatives $(x_i)_i$, rescale to
$\lambda_i x_i$, where $\lambda_0 := 1$ and
\begin{equation*}
   \lambda_{rn} \lambda_{(r+1)n} \omega_{rn, (r+1)n} = \mu^{(-1)^i}
\end{equation*}
for $0 < r < 2n+2$.  Check that this proves the first paragraph.

In the second paragraph we can only use further rescalings
preserving the subdiameters: $x_i \mapsto \delta^{(-1)^i} x_i$.
Such rescalings multiply $\prod_0^n \omega_{i, i+n+1}^{(-1)^i}$
by $\delta^{2n+2}$, so we can choose $\delta$ to make the product~$1$
if and only if $\prod_0^n \omega_{i, i+n+1}$ is positive.
To complete the proof, observe that $c_0 c_2 \cdots c_{n-1}$
is $\mu^{n+1} \prod_0^n \omega_{i, i+n+1}$.
There are two choices because the sign of $\delta$ is irrelevant.
\end{proof}

\section{Symmetric linear difference equations and the closure of $\cL_{n, N}(\K)$} \label{SLDEs}

In this section we present general results relating Lagrangian configurations
to non-degenerate symmetric linear difference equations of degree~$2n$.
The solution space of such an equation has a natural
symplectic form, generalizing the Wronski determinant.
When the equation has $N$-periodic coefficients and
monodromy~$-\Id$, there is a simple way to construct
a particular \Lc\ in its solution space.
This yields a projection from the space of all such equations
to equivalence classes of Lagrangian configurations.

\subsection{Linear difference operators}
\label{Lin Diff Ops}

Let $T$ be the {\it shift operator,\/} acting on infinite sequences
$(V_i)_{i\in\Z}$ by $(TV)_i := V_{i-1}$.
A {\em linear difference operator over\/}~$\K$
is a polynomial expression in $T$ and its inverse,
\begin{equation} \label{GenOp}
   A = a^n\, T^n + a^{n-1}\, T^{n-1} + \cdots + a^m\, T^m,
\end{equation}
where $m \le n$ are arbitrary integers and the coefficients~$a^\ell$
are sequences $(a^\ell_i)_{i\in\Z}$ of $\K$-scalars.
Such operators act on sequences $(V_i)$
of $\K$-scalars, the coefficients
acting by multiplication: $(a V)_i := a_i V_i$.

\begin{itemize}
\item
$A$ is said to be of {\em order\/} $n - m$
if both $a^m$ and $a^n$ are non-zero.

\smallbreak \item
$A$ is said to be {\em non-degenerate\/}
if both $a^m_i$ and $a^n_i$ are non-zero for all~$i$.

\smallbreak \item
$A$ is said to be {\em $N$-periodic\/}
if $a^\ell_i = a^\ell_{i+N}$ for all~$\ell$ and~$i$.

\end{itemize}

\begin{defn}
The {\it adjoint\/} $A^*$ of a linear difference operator $A$ is defined by
$$
   V \cdot (A^* W) = (AV) \cdot W,
$$
where $V \cdot W := \sum_{i\in\Z}V_i W_i$, an inner product
on scalar sequences with only finitely many non-zero terms.
\end{defn}

It is simple to check that $T^* = T^{-1}$.  This is the discrete analog
of the fact that translation is the exponential of the derivation
$\frac{d}{dx}$, and $\frac{d}{dx}^*=-\frac{d}{dx}$.
It is also clear that $(AB)^* = B^* A^*$ for any operators $A$ and $B$.
In particular, writing $(T^{\ell'} a^\ell)$ for the multiplication operator
$(T^{\ell'} a^\ell)_i = a^\ell_{i-\ell'}$, one obtains the following lemma.

\begin{lem}
$\bigl( \sum_{\ell = m}^n a^\ell\, T^\ell \bigr)^* \,=\,
\sum_{\ell = m}^n (T^{-\ell} a^\ell)\, T^{-\ell}$.
\end{lem}

\begin{defn}
If an operator $A$ satisfies $A = A^*$, it is
{\em self-adjoint,\/} or {\em symmetric.\/}
In this case, for some $n \ge 0$ there exist
sequences $a^0, \ldots, a^n$ such that
\begin{equation} \label{NashGenOp}
   A = a^0 + \sum_{\ell = 1}^n
   \bigl( a^\ell\,T^\ell + (T^{-\ell} a^\ell)\, T^{-\ell} \bigr).
\end{equation}
\end{defn}

\begin{rem}
The spectral theory of linear difference operators is quite similar to
that of linear differential operators; see~\cite{Kri} and references therein.
Operators with periodic or antiperiodic solutions play a special role in~\cite{Kri},
where they are called ``superperiodic''.
\end{rem}

\subsection{Linear difference equations}

The {\em linear difference equation\/} corresponding to
a linear difference operator $A$ is $AV = 0$.
We denote the space of solutions of this equation,
the kernel of $A$, by $\cK(A)$:
\begin{equation*}
   \cK(A) := \{ V:\, AV = 0 \}.
\end{equation*}

\begin{lem} \label{IC lemma}
Let $A$ be a non-degenerate linear difference operator over $\K$ of order~$p$.
For any~$i_0 \in \Z$ and any $\K$-scalars $c_{i_0+1}, \cdots, c_{i_0 + p}$,
there is a unique solution $(V_i)$ of the equation $A V = 0$
satisfying the initial conditions $V_i = c_i$ for $i_0 < i \le i_0 + p$.
In particular, $\cK(A)$ is a $p$-dimensional vector space over $\K$.
\end{lem}

The proof of this lemma is immediate.
Note that the symmetric operator~(\ref{NashGenOp})
is non-degenerate if and only if $a^n_i \not= 0$ for all~$i$.
Let us write the equation $AV = 0$ explicitly in this case:
\begin{equation} \label{NashGenEq}
   a^n_i\, V_{i-n} + \cdots + a^1_i\, V_{i-1} + a^0_i\, V_i +
   a^1_{i+1}\, V_{i+1} + \cdots + a^n_{i+n}\, V_{i+n} = 0\,
   \mbox{\rm\ for all $i$.}
\end{equation}

\begin{cor} \label{V^i(A)}
Given a non-degenerate symmetric difference operator $A$
over $\K$ of degree~$2n$ as in~(\ref{NashGenEq}),
for all $i \in \Z$ there is a unique element $V^i(A)$
of the kernel $\cK(A)$ such that
\begin{equation} \label{V^i(A) eqn}
   \bigl( V^i_{i-n}(A),\, V^i_{i-n+1}(A),\, \ldots,\,
   V^i_{i+n-1}(A),\, V^i_{i+n}(A) \bigr)
   \,:=\, \Bigl( -\frac{1}{a^n_i},\, 0,\, 0,\, \ldots,\,
   0,\, \frac{1}{a^n_{i+n}} \Bigr).
\end{equation}
For any~$i$, $\bigl\{ V^{i+1}(A),\, V^{i+2}(A),\,
\ldots,\, V^{i+2n}(A) \bigr\}$ is a basis of\/ $\cK(A)$.
\end{cor}

An important property of non-degenerate symmetric operators is the
existence of a natural symplectic form on their kernels.
Before giving the general result, we describe the simplest case.

\begin{ex}
The operator $\SLHS := T - a + T^{-1}$ is known as the discrete
Sturm-Liouville (or Hill, or Schr\"odinger) operator.
It is non-degenerate and symmetric,
and the classical {\em Wronski determinant\/}
$$
   \W(V, V') :=
   \biggl| \begin{array}{cc}
      V_{i-1} & V'_{i-1} \\[4pt]
      V_i & V'_i
   \end{array} \biggr|
$$
is a well-defined symplectic form on its kernel $\cK(\SLHS)$.
To understand this, check that when $\SLHS(V)$ and $\SLHS(V')$ are zero,
$\W(V, V')$ is independent of the choice of $i$:
$$
   \biggl| \begin{array}{cc}
      V_{i} & V'_{i} \\[4pt]
      V_{i+1} & V'_{i+1}
   \end{array} \biggr| \,=\,
   -\biggl| \begin{array}{cc}
      V_{i} & V'_{i} \\[4pt]
      V_{i-1} & V'_{i-1}
   \end{array} \biggr| -
      a_i \biggl| \begin{array}{cc}
      V_{i} & V'_{i} \\[4pt]
      V_{i} & V'_{i}
   \end{array} \biggr| \,=\, 
   \biggl| \begin{array}{cc}
      V_{i-1} & V'_{i-1} \\[4pt]
      V_i & V'_i
   \end{array} \biggr|.
$$
\end{ex}

\begin{rem}
The continuant~(\ref{ContEq})
may be viewed as an element of $\cK(\SLHS)$:
the Sturm-Liouville difference equation is
\begin{equation} \label{SLEq}
   V_{i-1} - a_iV_{i} + V_{i+1} = 0,
\end{equation}
and the initial conditions $(V_{-1},V_{0})=(0,1)$ give
$V_n = K_n(a_0, \ldots, a_{n-1})$.
In fact, continuants are the simplest members
of the series of {\em Andr\'e determinants,\/}
which satisfy linear difference equations of higher order;
see \cite{And} and also~\cite{SVRS}.
\end{rem}

We now define a multidimensional version of the Wronski determinant.
It is a discrete analog of the symplectic form on the solution space
of the symmetric linear differential equation studied in \cite{Ovs}.

\begin{defn}
Fix a non-degenerate symmetric linear difference operator $A$
over $\K$ of order~$2n$, as in~(\ref{NashGenOp}).
Given two elements $V$ and $V'$ of the kernel $\cK(A)$
and any $i \in \Z$, set
\begin{equation} \label{SymFEq}
   \W_A^i (V, V') := \sum_{\ell=1}^n \sum_{m=i+1}^{i+\ell}
   \, a^\ell_m\,
   \biggl| \begin{array}{cc}
      V_{m-\ell} & V'_{m-\ell} \\[4pt]
      V_m & V'_m
   \end{array} \biggr|.
\end{equation}
\end{defn}

\begin{lem} \label{SymLem}
\begin{enumerate}

\item[(i)]
$\W_A^i$ is independent of~$i$ and is a symplectic form $\W_A$ on $\cK(A)$.

\smallbreak \item[(ii)]
Writing $V^i$ for the solution $V^i(A)$ of Corollary~\ref{V^i(A)},
$\W_A(V^i, V^j) = V^i_j = -V^j_i$.

\smallbreak \item[(iii)]
In particular, $\W_A(V^i, V^j) = 0$ for $|i-j| < n$,
and $\W_A(V^{j-n}, V^j) = 1/a^n_j$.

\end{enumerate}
\end{lem}

\begin{proof}
Let us use the shorthand $\Bigl| {j \atop k} \Bigr|$
for $\Bigl| {V_j\ V'_j \atop V_k\ V'_k} \Bigr|$.
It is helpful to expand $\W_A^i (V, V')$ as
\begin{align*}
   & \textstyle
   \biggl( a^1_{i+1}\, \Bigl| {i \atop i+1} \Bigr| \biggr)
   + \biggl( a^2_{i+1}\, \Bigl| {i-1 \atop i+1} \Bigr|
   + a^2_{i+2}\, \Bigl| {i \atop i+2} \Bigr| \biggr)
   + \biggl( a^3_{i+1}\, \Bigl| {i-2 \atop i+1} \Bigr|
   + a^3_{i+2}\, \Bigl| {i-1 \atop i+2} \Bigr|
   + a^3_{i+3}\, \Bigl| {i \atop i+3} \Bigr| \biggr) \\[6pt]
   & \textstyle
   + \cdots + \biggl( a^n_{i+1}\, \Bigl| {i+1-n \atop i+1} \Bigr|
   + a^n_{i+2}\, \Bigl| {i+2-n \atop i+2} \Bigr|
   + \cdots + a^n_{i+n}\, \Bigl| {i \atop i+n} \Bigr| \biggr).
\end{align*}
To prove that $\W_A^i$ is independent of $i$,
verify that $\W_A^i (V, V') - \W_A^{i-1} (V, V')$ is
\begin{equation*} \textstyle
   \biggl( a^1_{i+1}\, \Bigl| {i \atop i+1} \Bigr|
   - a^1_i\, \Bigl| {i-1 \atop i} \Bigr| \biggr)
   + \biggl( a^2_{i+2}\, \Bigl| {i \atop i+2} \Bigr|
   - a^2_i\, \Bigl| {i-2 \atop i} \Bigr| \biggr)
   + \cdots + \biggl( a^n_{i+n}\, \Bigl| {i \atop i+n} \Bigr|
   - a^n_i\, \Bigl| {i-n \atop i} \Bigr| \biggr).
\end{equation*}
Convert $\Bigl| {i-\ell \atop i} \Bigr|$ to $-\Bigl| {i \atop i-\ell} \Bigr|$
and use~(\ref{NashGenEq}), $\Bigl| {i \atop i} \Bigr| = 0$,
and $AV = AV' = 0$ to check that this is
\begin{equation*}
   \biggl| \begin{array}{cc}
      V_i & V'_i \\[4pt]
      (AV)_i & (AV')_i
   \end{array} \biggr|
   = \biggl| \begin{array}{cc}
      V_i & V'_i \\[4pt]
      0 & 0
   \end{array} \biggr|
   = 0.
\end{equation*}
Thus we may write simply $\W_A$ for $\W_A^i$.
Clearly it is a skew-symmetric bilinear form on $\cK(A)$.

For~(ii) and~(iii), it suffices to check from the definitions that for any $V' \in \cK(A)$,
\begin{equation*}
   \W_A^{i-1} (V^i, V') =
   a^n_i\,
   \biggl| \begin{array}{cc}
   V^i_{i-n} & V'_{i-n} \\[4pt]
   V^i_i & V'_i
   \end{array} \biggr|
   = -V'_i.
\end{equation*}

To prove that $\W_A$ is non-degenerate, recall Lemma~\ref{Gram} and
consider the matrix $\Omega_{\W_A}$ of $\W_A$ in the basis
$\bigl\{ V^{j-n+1}, V^{j+1}, \ldots, V^{j+n} \bigr\}$: for $1 \le r, s \le 2n$,
$(\Omega_{\W_A})_{rs} := \W_A(V^{r+j-n}, V^{s+j-n})$.
The result will follow if we prove $\det (\Omega_{\W_A}) \not= 0$.
Applying~(iii), we find
\begin{equation*}
   \Omega_{\W_A} = 
   \biggl( \begin{array}{cc}
   0 & T \\[4pt]
   -T^t & 0
   \end{array} \biggr),
\end{equation*}
where $T$ is an $n \times n$ upper triangular matrix with diagonal entries
$(a^n_{j+1})^{-1}, (a^n_{j+2})^{-1}, \ldots, (a^n_{j+n})^{-1}$.
\end{proof}

\medbreak \noindent {\bf Rescaling.}
Suppose that $\lambda$ is a non-vanishing sequence over $\K$:
a sequence $(\lambda_i)_{i \in \Z}$ of non-zero $\K$-scalars.
Given an operator $A$, we define its {\em rescaling by\/} $\lambda$
to be the operator $\lambda^{-1} \circ A \circ \lambda^{-1}$.

\begin{lem} \label{RescaleLem}
Let $A$ be a non-degenerate symmetric linear difference operator
over $\K$ of order~$2n$, as in~(\ref{NashGenOp}),
and let $\lambda$ be a non-vanishing sequence over $\K$.
Let $\tilde A$ be the rescaling $\lambda^{-1} A \lambda^{-1}$.

\begin{enumerate}

\item[(i)]
$\tilde A$ is a non-degenerate symmetric
operator over $\K$ of order~$2n$.
Its coefficients $\tilde a^\ell$ are
\begin{equation*}
   \tilde a^\ell_i = \lambda_i^{-1} \lambda_{i-\ell}^{-1} a^\ell_i.
\end{equation*}

\smallbreak \item[(ii)]
If $A$ and $\lambda$ are $N$-periodic, then $\tilde A$ is too.

\smallbreak \item[(iii)]
$\lambda$ is a symplectic map from
$\bigl(\cK(A), \W_A \bigr)$ to $\bigl(\cK(\tilde A), \W_{\tilde A} \bigr)$.

\smallbreak \item[(iv)]
$\lambda \bigl( V^i(A) \bigr) = \lambda_i^{-1} V^i(\tilde A)$.

\end{enumerate}
\end{lem}

\begin{proof}
We leave~(i), (ii), and $\lambda \bigl(\cK(A)\bigr) = \cK(\tilde A)$ to the reader.
To prove $\lambda$ symplectic, verify
\begin{equation*}
   \W_{\tilde A}^i (\lambda V, \lambda V')
   := \sum_{\ell=1}^n \sum_{m=i+1}^{i+\ell}
   \, \tilde a^\ell_m \lambda_m \lambda_{m-\ell} \,
   \biggl| \begin{array}{cc}
      V_{m-\ell} & V'_{m-\ell} \\[4pt]
      V_m & V'_m
   \end{array} \biggr|.
\end{equation*}
Because $\tilde a^\ell_m \lambda_m \lambda_{m-\ell} = a^\ell_m$,
this is simply $\W_A^i (V, V')$.

For~(iv), use Corollary~\ref{V^i(A)} to check that
$\lambda \bigl( V^i(A) \bigr)_j = \lambda_i^{-1} V^i(\tilde A)_j$
for $i-n \le j \le i+n$.  By~(iii), both $\lambda \bigl( V^i(A) \bigr)$ and
$\lambda_i^{-1} V^i(\tilde A)$ are in $\cK(\tilde A)$,
so by Lemma~\ref{IC lemma} they are equal.
\end{proof}

\subsection{Periodic operators, monodromy, and Lagrangian configurations}

Difference equations corresponding to $N$-periodic operators
do not necessarily have $N$-periodic solutions.
However, we do have the following lemma.
Its proof is immediate from the obvious fact
that an operator is $N$-periodic
if and only if it commutes with $T^N$.

\begin{lem}
Suppose that $A$ is an $N$-periodic linear difference operator.
Then $T^N$ preserves the kernel $\cK(A)$.
It is called the {\em monodromy operator\/} $M_A$ of $A$:
\begin{equation*}
   M_A := T^N |_{\cK(A)}: \cK(A) \to \cK(A).
\end{equation*}
\end{lem}

In the case of non-degenerate symmetric operators,
the monodromy is symplectic:

\begin{lem} \label{monod symp}
Suppose that $A$ is a non-degenerate $N$-periodic symmetric
linear difference operator of order~$2n$.
Then the monodromy operator $M_A$ preserves
the symplectic form $\W_A$ on $\cK(A)$.
\end{lem}

\begin{proof}
We must prove that $\W_A(T^N V, T^N V') = \W_A(V, V')$
for all elements $V$ and $V'$ of $\cK(A)$.
Recall that $\W_A$ may be expressed as $\W_A^i$ for any~$i$.
Use the fact that $(T^N a^\ell) = a^\ell$ for all~$\ell$
to check that $\W_A^i(T^N V, T^N V') = \W_A^{i-N}(V, V')$.
\end{proof}

Our main result in Section~\ref{SLDEs} is Theorem~\ref{MonThm v2},
the most general result of the paper.
It states that a certain set of difference operators may be projected
to symplectic equivalence classes of Lagrangian configurations.
In order to define this projection we make two preliminary definitions.

\begin{defn}
For $N \ge 2n$, let $\ocL_{n, N}(\K)$ be the
$\Sp(2n, \K)$-moduli space of symplectic equivalence classes
of {\it all} $(n, N)$-\Lc s over $\K$, both generic and non-generic.
\end{defn}

\begin{defn}
For $N \ge 2n$, let $\DE_{n, N}(\K)$ be
the set of non-degenerate $N$-periodic
symmetric linear difference operators
over $\K$ of order~$2n$ with monodromy $-\Id$.
\end{defn}

\begin{rems}
\begin{itemize}

\item
Regarded as a subset of $(\K\bP^{2n-1})^N$,
$\cL_{n, N}(\K)$ is dense in $\ocL_{n,N}(\K)$
in both the standard and Zariski topologies.

\smallbreak \item
For both geometric and analytic reasons,
imposing the condition that the monodromy be $\Id$
in the definition of $\DE_{n, N}(\K)$ would be less natural;
cf.~\cite{SVRS, Kri} for the $\SL(2n)$-analog.

\smallbreak \item
Suppose that $A$ is a non-degenerate $N$-periodic
symmetric linear difference operator of order~$2n$.
Fix initial conditions $V_{i_0+1}, \ldots, V_{i_0+2n}$,
and let $V$ be the corresponding solution of $AV = 0$.
It is easy to see that each entry $V_i$ of $V$
depends polynomially on the quantities
$(a^n_k)^{\pm 1}, a^{n-1}_k, \ldots, a^0_k$ for $i_0 < k \le i_0+N$.
It follows that the same is true of $M_A$,
and so $\DE_{n, N}$ is an algebraic variety.

We will see that $\DE_{n, N}$ projects to
$\ocL_{n, N}$, with fibers given by rescaling.
Recall from Proposition~\ref{freely} that
$\ocL_{n, N}$ is $n(N-2n-1)$-dimensional.
The set of periodic rescalings has $N$ parameters,
so the dimension of $\DE_{n, N}$ is $(n+1)N - n(2n+1)$.
Thus the number of independent constraints imposed on a periodic
symmetric linear difference operator by specifying its monodromy
to be $-\Id$
is the dimension of the symplectic group preserving $\W_A$,
as one would predict from Lemma~\ref{monod symp}.

\end{itemize}
\end{rems}

\begin{prop} \label{P(A)}
Suppose that $A$ is in $\DE_{n, N}(\K)$.
Fix arbitrarily an identification of
the symplectic space $\bigl( \cK(A), \W_A \bigr)$
with the standard symplectic space $\bigl( \K^{2n}, \omega \bigr)$,
and let $v_i \in \K^{2n}$ be the image under this identification
of the element $V^i(A)$ of\/ $\cK(A)$ defined in
Corollary~\ref{V^i(A)}.

\begin{enumerate}

\item[(i)]
The $v_i$ are $N$-antiperiodic.

\smallbreak \item[(ii)]
$\bigl( \K v_1, \ldots, \K v_N \bigr)$ is an $(n, N)$-\Lc.

\smallbreak \item[(iii)]
There is a map $P: \DE_{n, N}(\K) \to \ocL_{n, N}(\K)$,
defined by
\begin{equation*}
P(A) :=  \mbox{\it the symplectic equivalence class of\/}\
\bigl( \K v_1, \ldots, \K v_N \bigr).
\end{equation*}

\smallbreak \item[(iv)]
$P(A)$ and $P(-A)$ are opposite configurations.

\end{enumerate}
\end{prop}

\begin{proof}
The fact that the monodromy $M_A$ is~$-\Id$ translates to
the statement that $V^{i+N}(A) = -V^i(A)$, giving~(i).
For~(ii), apply Lemmas~\ref{ii from i} and~\ref{SymLem}(iii)
and use the fact that $\W_A \bigl( V^i(A), V^j(A) \bigr)
= \omega(v_i, v_j)$ by construction.
For~(iii), note that the symplectic equivalence class
of $(\K v_1, \ldots, \K v_N)$ is independent of the choice
of symplectic identification of $\cK(A)$ with $\K^{2n}$.

For~(iv), use the facts that $\cK(-A) = \cK(A)$,
$V^i(-A) = -V^i(A)$, and $\W_{-A} = -\W_A$.
\end{proof}

\begin{thm} \label{MonThm v2}

\begin{enumerate}

\item[(i)]
$P: \DE_{n, N}(\K) \to \ocL_{n, N}(\K)$ is surjective.

\smallbreak \item[(ii)]
$P(A) = P(\tilde A)$ if and only if $\tilde A$ is a
rescaling $\lambda^{-1} A \lambda^{-1}$ of $A$
by an $N$-periodic $\lambda$.

\end{enumerate}
\end{thm}

\begin{proof}
We proceed by a series of lemmas.
For~(i), fix an $(n, N)$-\Lc\ $(\K x_1, \ldots, \K x_N)$.
As usual, extend the representatives $N$-antiperiodically
to $(x_i)_{i \in \Z}$ and write $\omega_{ij}$ for $\omega(x_i, x_j)$.
In order to construct an operator $A$ in $\DE_{n, N}(\K)$
such that $P(A)$ is the class of $(\K x_1, \ldots, \K x_N)$,
for all~$i$ define
\begin{equation} \label{a^n_i eqn}
   a^n_i := 1/\omega_{i-n, i}.
\end{equation}
Keeping in mind that $\{ x_{i-n+1}, \ldots, x_{i+n} \}$
is a basis of $\K^{2n}$, define $a^{n-1}_i, \ldots, a^{-n}_i$
by the equation
\begin{equation} \label{x SLDE v2}
   a^n_i x_{i-n} + a_i^{n-1} x_{i-n+1} + 
   \cdots + a_i^{-n+1} x_{i+n-1} + a_i^{-n} x_{i+n} = 0.
\end{equation}
Define $A$ by $(AV)_i := a^n_i V_{n-i} + \cdots + a^{-n}_i V_{n+i}$.
The next two lemmas concern this difference operator.

\begin{lem} \label{IndLem}
\begin{enumerate}

\item[(i)]
$A$ is non-degenerate, $N$-periodic, and symmetric.

\smallbreak \item[(ii)]
For $1 \le p \le n$, the coefficients $a^{n-p}_i$ are given by
\begin{equation} \label{a^{n-p}_i}
   a^{n-p}_i \,=\, \sum_{m=0}^{p-1}\,
   \sum_{0 < p_1 < \cdots < p_m < p} -(-1)^m\,
   \frac{\omega_{i-n, i+p_1}\, \omega_{i-n+p_1, i+p_2}
   \cdots \omega_{i-n+p_{m-1}, i+p_m}\, \omega_{i-n+p_m, i+p}}
   {\omega_{i-n, i}\, \omega_{i-n+p_1, i+p_1} \cdots
   \omega_{i-n+p_m, i+p_m}\, \omega_{i-n+p, i+p}},
\end{equation}
where the summand at $m=0$ is understood to be
$-\omega_{i-n, i+p} / \omega_{i-n, i}\, \omega_{i-n+p, i+p}$.

\end{enumerate}
\end{lem}

\begin{exs}
Observe that~(\ref{a^{n-p}_i}) has $2^{p-1}$ summands.
The first three cases are
\begin{align*}
   a^{n-1}_i \,=\, \frac{-1}{\omega_{i-n, i}}
   \Bigl( \frac{\omega_{i-n, i+1}}{\omega_{i-n+1, i+1}} & \Bigr), \\[4pt]
   a^{n-2}_i \,=\, \frac{-1}{\omega_{i-n, i}}
   \Bigl( \frac{\omega_{i-n, i+2}}{\omega_{i-n+2, i+2}} &
   - \frac{\omega_{i-n, i+1}\, \omega_{i-n+1, i+2}}
   {\omega_{i-n+1, i+1}\, \omega_{i-n+2, i+2}} \Bigr), \\[4pt]
   a^{n-3}_i \,=\, \frac{-1}{\omega_{i-n, i}}
   \Bigl( \frac{\omega_{i-n, i+3}}{\omega_{i-n+3, i+3}} &
   - \frac{\omega_{i-n, i+1}\, \omega_{i-n+1, i+3}}
   {\omega_{i-n+1, i+1}\, \omega_{i-n+3, i+3}}
   - \frac{\omega_{i-n, i+2}\, \omega_{i-n+2, i+3}}
    {\omega_{i-n+2, i+2}\, \omega_{i-n+3, i+3}} \\[4pt]
   & + \frac{\omega_{i-n, i+1}\, \omega_{i-n+1, i+2}\, \omega_{i-n+2, i+3}}
    {\omega_{i-n+1, i+1}\, \omega_{i-n+2, i+2}\, \omega_{i-n+3, i+3}} \Bigr).
\end{align*}
\end{exs}

\begin{proof}
Apply $\omega(x_{i-r}, \cdot)$ to~(\ref{x SLDE v2}) to obtain
\begin{equation} \label{omega solves A}
   0 = a^n_i \omega_{i-r, i-n} + a_i^{n-1} \omega_{i-r, i-n+1}
   + a_i^{n-2} \omega_{i-r, i-n+2} + \ldots + a_i^{-n} \omega_{i-r, i+n}.
\end{equation}
Consider the case $r = 0$.
By the Lagrangian condition, here only the leftmost and
rightmost terms on the right hand side are non-zero.
We obtain
\begin{equation} \label{a^n_{i+n} eqn}
   a^{-n}_i = 1/\omega_{i, i+n} = a^n_{i+n}.
\end{equation}
Thus $A$ is non-degenerate and satisfies the symmetry condition
$a^{-\ell}_i = a^\ell_{i+\ell}$ for $\ell = n$.

Now consider the cases $r = \pm p$ with $1 \le p \le n$.
By the Lagrangian condition, for $r = -p$ only the leftmost
$p+1$ terms on the right hand side are non-zero,
while for $r = p$ only the rightmost $p + 1$ terms are non-zero.
We obtain
\begin{align}
\label{a^{n-p}_i ind}
   & a^{n-p}_i = \frac{-1}{\omega_{i-n+p, i+p}}
   \Bigl( \frac{\omega_{i-n, i+p}}{\omega_{i-n, i}}
   + a^{n-1}_i\, \omega_{i-n+1, i+p} + \cdots
   + a^{n-p+1}_i\, \omega_{i-n+p-1, i+p} \Bigr), \\[4pt]
\label{a^{p-n}_i ind}
   & a^{p-n}_i = \frac{-1}{\omega_{i-p, i+n-p}}
   \Bigl( \frac{\omega_{i-p, i+n}}{\omega_{i, i+n}}
   + a^{1-n}_i\, \omega_{i-p, i+n-1} + \cdots
   + a^{p-1-n}_i\, \omega_{i-p, i+n-p+1} \Bigr).
\end{align}

A straightforward induction argument
from~(\ref{a^{n-p}_i ind}) gives~(\ref{a^{n-p}_i}):
the first term of~(\ref{a^{n-p}_i ind})
is the $m=0$ term of~(\ref{a^{n-p}_i}), and the term
$- a^{n-q}_i \omega_{i-n+q, i+p} / \omega_{i-n+p, i+p}$
of~(\ref{a^{n-p}_i ind}) gives those terms of~(\ref{a^{n-p}_i})
with $p_m = q$.
A parallel argument from~(\ref{a^{p-n}_i ind})
yields a closed formula for $a^{p-n}_i$:
\begin{equation} \label{a^{p-n}_i}
   a^{p-n}_i \,=\, \sum_{m=0}^{p-1}\,
   \sum_{0 < p_1 < \cdots < p_m < p} -(-1)^m\,
   \frac{\omega_{i-p, i+n-p_m}\, \omega_{i-p_m, i+n-p_{m-1}}
   \cdots \omega_{i-p_2, i+n-p_1}\, \omega_{i-p_1, i+n}}
   {\omega_{i-p, i+n-p}\, \omega_{i-p_m, i+n-p_m}
   \cdots \omega_{i-p_1, i+n-p_1}\, \omega_{i, i+n}},
\end{equation}
where the summand at $m=0$ is understood to be
$-\omega_{i-p, i+n} / \omega_{i-p, i+n-p}\, \omega_{i, i+n}$.

To finish proving that $A$ is symmetric, we must prove
$a^{-\ell}_i = a^\ell_{i+\ell}$ for $0 \le \ell < n$.  Note that
$$
   (p_1,\, p_2,\, \ldots,\, p_m) \mapsto
   (p-p_m,\, p-p_{m-1},\, \ldots,\, p-p_1)
$$
is an involution of the index set of the inner summation in~(\ref{a^{n-p}_i}).
Use this to verify that replacing $i$ by $i+n-p$ in~(\ref{a^{n-p}_i})
gives~(\ref{a^{p-n}_i}).  This completes the proof of the lemma:
the fact that $A$ is $N$-periodic is now immediate from
$\omega_{i+N, j+N} = \omega_{ij}$.
\end{proof}

\begin{lem} \label{V is omega}
\begin{enumerate}

\item[(i)]
$A$ lies in $\DE_{n, N}(\K)$.

\smallbreak \item[(ii)]
The solutions $V^i(A)$ defined in Corollary~\ref{V^i(A)}
are given by $V^i_j(A) = \omega_{ij}$.

\end{enumerate}
\end{lem}

\begin{proof}
We begin with~(ii).  Abbreviate $V^i(A)$ by $V^i$.
By~(\ref{V^i(A) eqn}), (\ref{a^n_i eqn}), and~(\ref{a^n_{i+n} eqn}),
\begin{align*}
   \bigl( V^i_{i-n},\, V^i_{i-n+1},\, \ldots,\,
   V^i_{i+n-1},\, V^i_{i+n} \bigr)
   & \,:=\, \Bigl( \omega_{i, i-n},\, 0,\, \ldots,\,
   0,\, \omega_{i, i+n} \Bigr) \\[4pt]
   & \,:=\, \Bigl( \omega_{i, i-n},\, \omega_{i, i-n+1} \ldots,\,
   \omega_{i, i+n-1},\, \omega_{i, i+n} \Bigr).
\end{align*}
Consider~(\ref{omega solves A}): since $i-r$ is arbitrary,
we see that $(\omega_{ij})_j$ lies in $\cK(A)$.
By the above identity, it has the same initial conditions
as $V^i$, and so~(ii) follows from Lemma~\ref{IC lemma}.

In light of Lemma~\ref{IndLem}(i), to prove~(i) it suffices
to prove that~(\ref{x SLDE v2}) has monodromy $-\Id$.
Because the $V^i$ span $\cK(A)$, this reduces to
$V^i_{j+N} = -V^i_j$ for all $i,\, j$.
By~(ii), this follows from $x_{j+N} = -x_j$.
\end{proof}

At this point we have proven Theorem~\ref{MonThm v2}(i):
by Lemmas~\ref{SymLem} and~\ref{V is omega},
the element $A$ of $\DE_{n, N}(\K)$ constructed in
Lemma~\ref{IndLem} has solutions $V^i(A)$
satisfying $\W_A(V^i, V^j) = V^i_j = \omega(x_i, x_j)$.
Therefore by Lemma~\ref{Gram}(iii) there is an element 
of $\Sp(2n, \K)$ carrying the $x_i$ to the $v_i$ of
Proposition~\ref{P(A)}, and so $P(A)$ is the class
of the Lagrangian configuration originally given.

We now turn to Theorem~\ref{MonThm v2}(ii).
The fact that $P(A) = P(\tilde A)$ if 
$A = \lambda \tilde A \lambda$ is immediate
from Lemma~\ref{RescaleLem}: $\lambda$
is a symplectic map carrying $V^i(A)$ to a multiple
of $V^i(\tilde A)$.
Conversely, suppose that $P(A) = P(\tilde A)$.
Reviewing Proposition~\ref{P(A)}, we find that
this means there is a symplectic map
$\Lambda$ from $\bigl( \cK(A), \W_A \bigr)$
to $\bigl( \cK(\tilde A), \W_{\tilde A} \bigr)$
carrying $V^i(A)$ to a non-zero multiple of
$V^i(\tilde A)$, for all~$i$.
Define $\lambda$ by setting
$\lambda_i^{-1}$ to be this multiple.
Because the sequences $V^i(A)$ and
$V^i(\tilde A)$ are both $N$-antiperiodic,
$\lambda$ is $N$-periodic.

Let $\hat A := \lambda \tilde A \lambda$.
By Lemma~\ref{RescaleLem}, $\lambda$
is a symplectic map from $\bigl( \cK(\hat A), \W_{\hat A} \bigr)$
to $\bigl( \cK(\tilde A), \W_{\tilde A} \bigr)$
carrying $V^i(\hat A)$ to $\lambda_i^{-1} V^i(\tilde A)$.
Therefore $\lambda^{-1} \circ \Lambda$
is a symplectic map from $\bigl( \cK(A), \W_A \bigr)$
to $\bigl( \cK(\hat A), \W_{\hat A} \bigr)$
carrying $V^i(A)$ to $V^i(\hat A)$.
The following lemma shows that $A = \hat A$,
completing the proof of Theorem~\ref{MonThm v2}.
\end{proof}

\begin{lem} \label{A is hat A}
Let $A$ and $\hat A$ be elements of $\DE_{n, N}(\K)$.
Write $a^\ell$ and $\hat a^\ell$ for the coefficients of $A$ and $\hat A$,
$V^i$ and $\hat V^i$ for $V^i(A)$ and $V^i(\hat A)$, and
$\nu_{ij}$ and $\hat \nu_{ij}$ for the symplectic products
$\W_A(V^i, V^j)$ and $\W_{\hat A}(\hat V^i, \hat V^j)$, respectively.
The following statements are equivalent:

\begin{enumerate}

\item[(i)]
$A = \hat A$, i.e., $a^\ell = \hat a^\ell$ for\/ $0 \le \ell \le n$.

\smallbreak \item[(ii)]
There exists a symplectic map $\sigma:
\bigl( \cK(A), \W_A \bigr) \to
\bigl( \cK(\hat A), \W_{\hat A} \bigr)$
such that $\sigma(V^i) = \hat V^i$ for all~$i$.

\smallbreak \item[(iii)]
$\nu_{ij} = \hat \nu_{ij}$
for all~$i$ and~$j$.

\smallbreak \item[(iv)]
$V^i = \hat V^i$ for all~$i$.

\end{enumerate}
\end{lem}

\begin{proof}
It is immediate that (i) implies (ii), (iii), and~(iv),
and (ii) implies (iii).  By Lemma~\ref{Gram}(iii),
(iii) implies (ii), and~(iii) and~(iv)
are equivalent by Lemma~\ref{SymLem}(ii).
In order to prove that~(iii) and~(iv) imply~(i),
we will prove that for $0 \le p \le n$ the $a^{n-p}_i$
are given by~(\ref{a^n_i eqn}) and~(\ref{a^{n-p}_i})
with $\nu_{ij}$ replacing $\omega_{ij}$.

For $p = 0$, recall that by Lemma~\ref{SymLem}(iii),
$a^n_i = 1/\nu_{i-n, i}$.  For $p > 0$ it suffices to prove
that~(\ref{a^{n-p}_i ind}) holds with $\nu$ replacing $\omega$,
and for this it suffices to prove that~(\ref{omega solves A})
holds with $\nu$ replacing $\omega$.
This identity in turn results from applying $\W_A(V^{i-r}, \cdot)$
to~(\ref{x SLDE v2}) with $V^\bullet$ replacing $x_\bullet$,
so finally we come down to proving the vector identity
\begin{equation*}
   a^n_i V^{i-n} + a_i^{n-1} V^{i-n+1} + 
   \cdots + a^{n-1}_{i+n-1} V^{i+n-1} + a^n_{i+n} V^{i+n} = 0
\end{equation*}
for all~$i$.  Because $V^j$ is itself in $\cK(A)$, we know that
the scalar identity
\begin{equation*}
   a^n_i V^j_{i-n} + a_i^{n-1} V^j_{i-n+1} + 
   \cdots + a^{n-1}_{i+n-1} V^j_{i+n-1} + a^n_{i+n} V^j_{i+n} = 0
\end{equation*}
holds for all~$i$.  To complete the proof, recall from
Lemma~\ref{SymLem}(ii) that $V^j_k = -V^k_j$. 
\end{proof}

\section{The case $N = 2n + 3$} \label{Remarks}

In this section, let $(X_1, \ldots, X_N)$ be a generic complex $(n, N)$-\Lc\
with representatives $(x_1, \ldots, x_N)$.  Extend them to an
$N$-antiperiodic sequence and write $\omega_{ij}$ for $\omega(x_i, x_j)$.
We conclude the article with a discussion of the case $N = 2n + 3$:
we generalize the five Gauss relations~(\ref{Penta})
on $\cL_{1, 5}(\C)$ to $2n + 3$ relations on the $2n + 3$ basic
symplectic cross-ratios of $\cL_{n, 2n+3}(\C)$.

These relations are obtained by means of the
symmetric linear difference operators associated
to \Lc s in Theorem~\ref{MonThm v2}.
The computations actually consist in solving the system
of equations given by the condition that the
operators have monodromy $-\Id$.
We remark that Theorem~\ref{EquivThm}(i)
may be obtained via the same method.

There are two sequences of non-trivial cross-ratios
on $\cL_{n, 2n+3}$: the $c_i$ of~(\ref{L(n, 2n+2) CRs}),
and the $\gamma_{i, i+n+1}$ of~(\ref{gammas}),
which we will abbreviate by $\gamma_i$.
Both are $(2n + 3)$-periodic:
\begin{equation*}
   c_i := c_{i, i+1, i+n+1, i+n+2} =
   \frac{\omega_{i, i+n+1}\, \omega_{i+1, i+n+2}}
   {\omega_{i, i+n+2}\, \omega_{i+1, i+n+1}},
   \quad
   \gamma_i := c_{i-2, i, i+n, i+n+1}
   = \frac{\omega_{i-2, i+n}\, \omega_{i, i+n+1}}
   {\omega_{i-2, i+n+1}\, \omega_{i, i+n}}.
\end{equation*}

We have $\gamma_i = c_{i-1} c_{i+n}$,
so all cross-ratios on $\cL_{n, 2n+3}$
may be written in terms of the $c_i$.
We will prove that over $\C$ these $2n + 3$ cross-ratios
determine the equivalence class of their \Lc\
(we expect that over $\R$ they determine it up to opposites).
By Proposition~\ref{freely}, $\cL_{n, 2n+3}(\C)$ is $2n$-dimensional,
so the space of relations on the $c_i$ must have Krull dimension~3.
We conjecture that the $2n+3$ relations
we present generate the full space of relations.

If $\omega_{i, i+n} = 1$ for all~$i$, as in Section~\ref{Even Complex Norms},
then the representatives $x_i$ are said to be {\em normalized.\/}
We begin with a general lemma permitting us to restrict
our consideration to such representatives.

\begin{lem}
\label{nN norm}
For any complex $(n, N)$-\Lc\ $(X_1, \ldots, X_N)$
with $N/\GCD(n, N)$ odd, there exist exactly $2^{\GCD(n, N)}$
normalized choices of representatives.
\end{lem}

\begin{proof}
Following Section~\ref{MTii}, let $(\tilde x_1, \ldots, \tilde x_N)$
be any representatives of $(X_1, \ldots, X_N)$,
with corresponding symplectic products $\tilde \omega_{ij}$.
Fix an $N$-periodic sequence $\chi_i$
such that $\chi_i^2 = \tilde \omega_{i, i+n}$.
Mimicking~(\ref{scale factors}), set
\begin{equation*}
   \lambda_i \,:=\, \frac{\chi_i \chi_{i+2n} \chi_{i+4n} \cdots \chi_{i+(q-1)n}}
   {\chi_{i+n} \chi_{i+3n} \cdots \chi_{i+(q-2)n}} \,=\,
   \prod_{r=0}^{q-1} \chi_{i+rn}^{(-1)^r},
\end{equation*}
where $q$ denotes $N/\GCD(n, N)$.
Check that $\lambda_i \lambda_{i+n} = \tilde \omega_{i, i+n}$,
so $x_i := \tilde x_i / \lambda_i$ satisfies $\omega_{i, i+n} = 1$.

The fact that $\lambda_i^2 = \prod_{r=0}^{q-1}
\tilde \omega_{i+rn, i+ (r+1)n}^{(-1)^r}$ implies that each $q$-tuple
$(\lambda_i, \lambda_{i+n}, \ldots, \lambda_{i+(q-1)n})$
is determined up to a single choice of overall sign.
The lemma follows.
\end{proof}

Henceforth let $(X_1, \ldots, X_{2n+3})$ be a generic complex
$(n, 2n+3)$-\Lc, and fix a normalized choice of representatives $x_i$.
Recall from Section~\ref{Even Complex Norms} the
$(n+1)$-periodic sequence $a_i := \omega_{i, i+n+1}$ of
symplectic diameters of normalized $(n, 2n+2)$-configurations.
The analog here is the $(2n+3)$-periodic sequence of
{\em symplectic main diagonals,\/} defined by
the same formula as the $a_i$:
\begin{equation*}
   d_i := \omega_{i, i+n+1}.
\end{equation*}
Note that this notation is consistent with~(\ref{Penta}), and
\begin{equation*}
   c_i = d_i d_{i+1} / d_{i-n-1}, \qquad
   \gamma_i = c_{i-1} c_{i+n} = d_i d_{i+n}.
\end{equation*}

\begin{cor} \label{norm 2n+3}
\begin{enumerate}

\item[(i)]
If\/ $3 \notdivides n$, then $(X_1, \ldots, X_{2n+3})$ has two
normalized choices of representatives: $(x_i)_i$ and $(-x_i)_i$.
Both have the same $d_i$: the configuration determines its main diagonals.

\smallbreak \item[(ii)]
If\/ $3 \divides n$, then $(X_1, \ldots, X_{2n+3})$ has eight
normalized choices of representatives: $\epsilon_i x_i$, where
$\epsilon_i = \pm 1$ and depends only on $i$ modulo~3.
The corresponding main diagonals are $\epsilon_i \epsilon_{i+1} d_i$.

\end{enumerate}
\end{cor}

\begin{prop}
Generic equivalence classes in $\cL_{n, 2n+3}(\C)$
with the same cross-ratios $c_i$ are equal.
\end{prop}

\begin{proof}
Let $x_i$ and $\tilde x_i$ be normalized representatives of two
\Lc s having the same cross-ratios: $\tilde c_i = c_i$.
It suffices to show that the $\tilde x_i$ may be chosen so that
the two sets of main diagonals are the same, i.e., $\tilde d_i = d_i$,
as then $\tilde \omega_{ij} = \omega_{ij}$ for all~$i$ and~$j$.
Observe that
\begin{equation*}
   c_{i-1} \gamma_{i+1} = d_{i-1} d_i d_{i+1},
   \mbox{\rm\ and for any $r$,\ \ }
   \frac{\gamma_i \gamma_{i+2n} \cdots \gamma_{i+2rn}}
   {\gamma_{i+n} \gamma_{i+3n} \cdots \gamma_{i+(2r-1)n}}
   = d_i d_{i+(2r+1)n}.
\end{equation*}

Write $n$ in the form $3m + s$, where $s \in \{ -1, 0, 1 \}$.
Take $r = m$ above and apply $(2n+3)$-periodicity to see
that $d_i d_{i + s}$ is determined by the cross-ratios.
If $s = \pm 1$, dividing $c_{i-1} \gamma_{i+1}$ by this
gives $d_{i - s}$ as a function of the cross-ratios.
Thus for $3 \notdivides m$, the cross-ratios determine
the main diagonals.

If $3 \divides n$, i.e., $s = 0$, taking $r = m$ gives $d_i^2$,
and so the cross-ratios determine the main diagonals up to sign:
$\tilde d_i = \delta_i d_i$ for some $\delta_i = \pm 1$.
Using $\gamma_i = d_i d_{i+n}$,
$\tilde\gamma_i = \gamma_i$,
and $\GCD(n, 2n+3) = 3$,
we find that $\delta_i$ depends only on $i$ modulo~3.
Applying $c_{i-1} \gamma_{i+1} = d_{i-1} d_i d_{i+1}$,
we obtain $\delta_{i-1} \delta_i \delta_{i+1} = 1$.
Therefore by Corollary~\ref{norm 2n+3} it is possible
to modify the $\tilde x_i$ so that $\tilde d_i = d_i$.
\end{proof}

Recall now the difference operator $A$ constructed from
the representatives $x_i$ in~(\ref{x SLDE v2}).
In the normalized case the formula~(\ref{a^{n-p}_i})
for its coefficients simplifies, as the denominators are all~$1$:
\begin{equation*}
   a^{n-p}_i \,=\, \sum_{m=0}^{p-1}\,
   \sum_{0 < p_1 < \cdots < p_m < p} -(-1)^m\,
   \omega_{i-n, i+p_1}\, \omega_{i-n+p_1, i+p_2}
   \cdots \omega_{i-n+p_{m-1}, i+p_m}\, \omega_{i-n+p_m, i+p},
\end{equation*}
the summand at $m=0$ being $-\omega_{i-n, i+p}$.
Observe that
\[
  \omega_{i-n+p_{\ell-1}, i+p_\ell} = 
  \begin{cases} 
   d_{i-n+p_\ell -1} & \text{if } p_\ell - p_{\ell-1} = 1, \\
   d_{i+p_\ell} & \text{if } p_\ell - p_{\ell-1} = 2, \\
   1 & \text{if } p_\ell - p_{\ell-1} = 3, \\
   0 & \text{if } p_\ell - p_{\ell-1} \ge 4.
  \end{cases}
\]
We write $a^{n-p}_i$ explicitly for $0 \le p \le 3$:
\begin{align*}
   & a^n_i = 1, \qquad
   a^{n-1}_i = -d_{i-n}, \qquad
   a^{n-2}_i = d_{i-n} d_{i-n+1} - d_{i+2}, \\[4pt]
   & a^{n-3}_i = -d_{i-n} d_{i-n+1} d_{i-n+2}
   + d_{i-n} d_{i+3} + d_{i-n+2} d_{i+2} - 1.
\end{align*}

\begin{prop}
The $2n + 3$ symplectic main diagonals $d_i$ of
the normalized representatives $x_i$ of
$(X_1, \ldots, X_{2n+3})$ satisfy the following
$2n + 3$ polynomial relations
(for $n=1$, take $a^2_i$ to be~$0$):
\begin{equation} \label{2n+3 relation}
   0 = a^2_i + d_{i+n+1} a^1_i + d_i a^0_i + a^1_{i+1}.
\end{equation}
\end{prop}

\begin{proof}
Simply take $r = -(n+1)$ in~(\ref{omega solves A}).
\end{proof}

\medbreak \noindent {\bf Legendrian pentagons in $\CP^1$.}
This is $\cL_{1, 5}$, the Gaussian case
discussed in Section~\ref{Gauss relations}.
Here (\ref{2n+3 relation}) reduces to~(\ref{Penta}).
Using $d_i = c_{i-2} c_{i+1} / c_{i+2}$, the relations
may be stated in terms of the $c_i$:
\begin{equation*}
   \frac{1}{c_i} + \frac{1}{c_{i-1} c_{i+1}} = 1.
\end{equation*}

\medbreak \noindent {\bf Legendrian heptagons in $\CP^3$.}
For $\cL_{2, 7}$, (\ref{2n+3 relation}) reads
\begin{equation*}
   d_{i-1} d_i d_{i+1} - d_{i-3} d_{i-1} - d_{i+1} d_{i+3} - d_i + 1 = 0.
\end{equation*}
Using $d_i = c_{i-1} c_i c_{i+3} / c_{i-2} c_{i+1}$, this becomes
\begin{equation*}
   \frac{1}{c_{i-3}} + \frac{1}{c_{i+3}} + \frac{1}{c_{i-2} c_{i+2}}
   - \frac{1}{c_{i-3} c_i c_{i+3}} = 1.
\end{equation*}

\medbreak \noindent {\bf Legendrian nonagons in $\CP^5$.}
For $\cL_{3, 9}$, (\ref{2n+3 relation}) yields
\begin{align*}
   d_{i-4} d_{i-3} d_{i+3} d_{i+4}
   - d_{i-4} d_{i-3} d_{i-1} - d_{i-3} d_i d_{i+3}
   - d_{i+1} d_{i+3} d_{i+4} & \\[4pt]
   - d_{i-4} d_{i+4} + d_{i-1} d_{i+1}
   + d_{i-3} + d_i + d_{i+3} &= 0.
\end{align*}
Using $c_i = d_i d_{i+1} / d_{i-4}$, $c_{i-1} c_{i+3} = \gamma_i = d_i d_{i+3}$,
and $c_{i-1} \gamma_{i+1} = d_{i-1} d_i d_{i+1}$, this may be rewritten as
\begin{equation*}
   \frac{1}{c_{i-1}} + \frac{1}{c_i} + \frac{1}{c_{i+1}}
   + \frac{1}{c_{i-1} c_{i+1}} + \frac{1}{c_{i-2} c_{i+2}}
   - \frac{1}{c_{i-4} c_i c_{i+1}} - \frac{1}{c_{i-2} c_i c_{i+2}}
   - \frac{1}{c_{i-1} c_i c_{i+4}} = 1.
\end{equation*}

We close with a few general remarks.
Note that for $3 \notdivides n$, (\ref{2n+3 relation}) can always
be written as a rational relation on the $c_i$, because the $d_i$
are rational functions of the $c_i$.  In light of the situation
for $\cL_{3,9}$, we expect that this is in fact true for all~$n$.
Also, although we have worked only over $\C$ in this section,
it should be easy to show that the relations we have given
on the $c_i$ hold also over $\R$.

Finally, let us reiterate our conjecture regarding~(\ref{2n+3 relation}).
Because $\cL_{n, 2n+3}$ is $2n$-dimensional and the cross-ratios
$(c_1, \ldots, c_{2n+3})$ form a coordinate ring on it,
the space of relations on the $c_i$ must be of Krull dimension~3.
We conjecture that the $2n + 3$ relations~(\ref{2n+3 relation})
generate the full relation space.

\bigskip \noindent
{\bf Acknowledgements}.
We are grateful to Sophie Morier-Genoud, Sergei Tabachnikov,
and Richard Schwartz for enlightening discussions.
C.H.C.\ was partially supported by Simons
Collaboration Grants 207736 and 519533.

\end{document}